\documentclass{article}
\usepackage[utf8]{inputenc}
\usepackage{amssymb}
\usepackage[pdftex]{graphicx}
\usepackage{a4wide}
\usepackage{amsmath}
\usepackage{euscript}
\usepackage{amsthm}
\usepackage{mathtools}
\usepackage{amsopn}
\usepackage{stackrel}
\usepackage[all]{xy}
\usepackage{setspace}
\DeclareGraphicsExtensions{.pdf,.png,.jpg}
\usepackage{float}
\usepackage{blkarray}
\usepackage[pdftex, colorlinks, linkcolor=blue, citecolor=blue, urlcolor=blue]{hyperref}
\usepackage{mathabx}
\usepackage{ytableau}
\usepackage{pgfplots}
\pgfplotsset{width=7cm, compat=1.10}
\usepgfplotslibrary{fillbetween}
\usepackage[ruled,vlined]{algorithm2e}
\usetikzlibrary{matrix}
\usepackage{mathdots}

\newcommand{\R}{\mathbb{R}}

\newcommand{\C}{\mathbb{C}}

\newcommand{\Sym}{\mathcal{S}}

\newcommand{\PSD}{\mathrm{PSD}}

\newtheorem{theorem}{Theorem}[section]
\newtheorem{cor}{Corollary}[section]

\newtheorem{prop}{Proposition}[section]
\theoremstyle{definition}

\newtheorem{definition}{Definition}[section]
\newtheorem{ex}{Example}[section]

\newtheorem{remark}{Remark}
\newtheorem*{continuex}{Example \continueref \space Continued}
\newenvironment{continue}[1]
  {\newcommand\continueref{\ref{#1}}\continuex}
  {\endcontinuex}

\title{Symmetry Adapted Gram Spectrahedra}
\author{Alexander Heaton\footnote{Max Planck Institute for Mathematics in Sciences, Leipzig, and TU Berlin; {\tt heaton@mis.mpg.de}}, Serkan Ho\c sten\footnote{Mathematics Department, San Francisco State University; {\tt serkan@sfsu.edu}}, and Isabelle Shankar\footnote{Mathematics Department, UC Berkeley; {\tt isabelle\_shankar@berkeley.edu}} }

\begin{document}

\maketitle

\begin{abstract}
    This paper explores the geometric structure of the spectrahedral cone, called the symmetry adapted PSD cone, and the symmetry adapted Gram spectrahedron of a symmetric polynomial.  In particular, we determine the dimension of the symmetry adapted PSD cone, describe its extreme rays, and discuss the structure of its matrix representations.  We also consider the symmetry adapted Gram spectrahedra for specific families of symmetric polynomials including binary symmetric polynomials, quadratics, and ternary quartics and sextics which give us further insight into these symmetric SOS polynomials. Finally, we discuss applications of the theory of sums of squares and symmetric polynomials which arise from symmetric function inequalities.  
\end{abstract}

\section{Introduction}\label{section:introduction}

We study the spectrahedra that arise in the theory of symmetric and sums of squares (SOS) polynomials. For a finite group $G$, we are interested in sums of squares polynomials which are $G$-invariant. We start with a representation of $G$ on $\R^n$, extending by linear substitution to a representation $D:G \to GL(V)$ on $V = \R[x_1,\dots,x_n]_d$, the vector space of degree $d$ homogeneous polynomials in $n$ indeterminates. The dimension
of $V$ is $N = \binom{n+d-1}{d}$, and we denote the cone of $N \times N$ positive semidefinite matrices by $\PSD_N$. Choosing a basis for $V$ gives matrices $D(g)$, and we obtain the {\it symmetry adapted} version of $\PSD_N$, namely,
\begin{equation*}
    \PSD_N^G := \bigg\{ Q \in \PSD_N \hspace{0.3cm} \bigg\rvert \hspace{0.3cm} D(g)^T Q D(g) = Q, \text{ for all }g \in G \bigg\}.
\end{equation*}
 We give precise definitions in Section \ref{section:preliminaries} below, but briefly, for a given polynomial $f$ of degree $2d$ which is invariant under the action of a group $G$, the \textit{symmetry adapted Gram spectrahedron} of $f$ is the closed, convex, semi-algebraic set
\begin{equation*}
    K_f^G := L_f \cap \PSD_N^G.
\end{equation*}
Here, $L_f$ is the linear space of symmetric matrices $Q$ which represent $f$ as $f(x) = m(x)^T Q m(x)$, and $m(x)$ is a column vector whose entries form a basis for $V$, usually chosen to be all monomials of degree $d$ in the variables $x_1,\dots,x_n$.

The \textit{Gram spectrahedron} $K_f = L_f \cap \PSD_N$ for a polynomial $f$ is a set parameterizing all ways to write $f$ as a sum of squares. Its geometry is important for understanding sums of squares representations of $f$. For example, the matrices of lowest rank contained in $K_f$ encode the ways to write $f$ as a sum of a minimal number of squares. These matrices of lowest rank are extremal points of $K_f$. Characterizing the minimal number of squares  
is a topic that has been widely studied \cite{CasselsEllisonPfister1971, ChoiLamReznick1995sumsofsquaresofrealpolys, GramSpectrahedra2017chuaplaumannsinnvinzant, 2019linearalgebramethod,  Robinson1973somepolysnotsumofsquares, Scheiderer2017sumofsquareslengthofrealforms,  Yiu2001sumofsquareslength}. 
The symmetry adapted Gram spectrahedron $K_f^G$ is a smaller and simpler convex set for which we can ask similar questions.
It was introduced in \cite{GP2004} and has since been used in a variety of applications \cite{BGSV2012, RSST2018, RST2018}.

In this paper, we mainly focus on the case $G=S_n$, the symmetric group, and the polynomials we consider will be the usual symmetric polynomials \cite{Macdonald}. Section \ref{section:preliminaries} offers a brief summary of the background needed from representation theory and SOS polynomials.  Interestingly, the sum of squares for a $G$-invariant SOS polynomial of degree $2d$ is itself composed of invariant partial sums, one from each isotypic component appearing in the $G$-representation on polynomials of degree $d$.  In Section \ref{section:properties-of-symm-adapted-PSD-cone} we go on to focus on the symmetry adapted cone $\PSD_N^G$. In particular, we compute the dimension of $\PSD_N^G$, characterize its extremal rays, and in the case of $G=S_n$, we present the block in any symmetric matrix $Q\in \PSD_N^{S_n}$ corresponding to the trivial representation. 
Section \ref{section:binary-quadratic} collects
our results on binary and quadratic symmetric polynomials that are SOS. In the binary case, we compute the symmetry adapted 
matrix representations of all symmetric polynomials, and in the quadratic case, we do the same, and prove that, as the number of indeterminates tends to infinity, the ratio of SOS symmetric quadratic forms to all symmetric quadratic forms is $\frac{1}{8}$. Another interesting consequence obtained is that symmetric quadratic SOS polynomials in $n$ variables can only be sums of $1$, $n-1$ or $n$ squares. 
In Section \ref{section:ternary-symmetric}, we start with the classic case of ternary quartics, describing the associated symmetry adapted PSD cone.  We then completely describe the geometric structure of the symmetry adapted Gram spectrahedron for a generic, smooth, positive, symmetric ternary quartic including the rank of the matrices on its boundary.  Further, we provide necessary conditions on the coefficients for a symmetric ternary quartic to be SOS.  
We continue the section by going up in degree and considering degree six symmetric polynomials in three variables.  Here we show 
that the rank of a matrix in the symmetry adapted Gram spectrahedron of a generic symmetric ternary sextic will be at least $4$. 
We end with Section \ref{sec:inequalities} where we consider an application of the SOS machinery to symmetric polynomial inequalities.  Included are three posets on partitions of $8$, $9$, and $10$  which represent SOS certifications on the difference of term-normalized homogeneous polynomials on the nonnegative orthant in $\R^3$.  
These results indicate many explicit counterexamples to Conjecture 7.2 in \cite{CGS2011}.

\section{Preliminaries}\label{section:preliminaries}

\subsection{Representation theory and symmetry adapted bases}

 A representation of a group $G$ is a homomorphism $ \rho: G \to GL(V)$ where $GL(V)$ is the group of invertible linear transformations of a vector space $V$. If $V$ is finite-dimensional, we also write $GL(n)$ for $n = \mathrm{dim} \, V$. 
 A subrepresentation of $V$ is a subspace $U \subset V$ which is invariant under the action of $G$.
If the only subrepresentations of $V$ are $\{0\}$ and $V$, we say that $V$ is irreducible. The character $\chi_\rho:G \to \mathbb{C}$ is defined by taking the trace of each $\rho(g)$ and is used to decompose representations. A representation which admits a direct sum decomposition $V = \oplus \, V_i$ with each $V_i$ irreducible is said to be completely reducible. 
Representations of finite groups are completely reducible. When we decompose $V$ into irreducibles $V_1, \ldots, V_s$, each $V_i$ appears with multiplicity~$m_i$:
\begin{equation*}
    V = m_1 V_1 \oplus \cdots \oplus m_s V_s.
\end{equation*}
This means that there exists a basis of $V$ such that $\rho(g)$ becomes the matrix $D(g)$ for $g \in G$ and is block diagonal with
$m_i$ blocks corresponding to $V_i$ where each block is 
$n_i \times n_i$ with $n_i = \dim V_i$. Here we denote by $D(g)$ the matrix written in a chosen basis for the linear map $\rho(g)$.

In general, these $m_i$ matrices of size $n_i \times n_i$ corresponding to $V_i$ are {\it not} identical. Fortunately, one can choose a different basis of $V$ with respect to which the representation matrices $\tilde{D}(g)$ for all $g \in G$ are block
diagonal where the $m_i$ blocks corresponding to $V_i$ are identical. See \cite[Section 5.2]{FS1992} or \cite[p. 23]{Serre1977linearRepresentationsOfFiniteGroups} for Algorithm \ref{alg:symmetry-adapted-basis}
to compute such a basis. In other words, the algorithm constructs a change of basis matrix $T$ such that $T^{-1} D(g) T$ is block diagonal with this extra nice property for all $g \in G$. Such a basis is known
as a \emph{symmetry adapted basis}. A symmetry adapted basis can also be used to simplify linear operators $P \in \text{Hom}(V,V)$ which commute
with the representation matrices $D(g)$ for all $g\in G$. 

In this paper, we are concerned with the field of real numbers $\R$, but to more easily and uniformly describe the representation theory involved we work with $\C$. Irreducible representations of finite groups over $\C$ come in three types \cite[p. 108]{Serre1977linearRepresentationsOfFiniteGroups}. All of them give rise to representations of $G$ over $\R$, although the dimension may stay the same (type 2) or double (types 1 or 3). The characters of the representations over $\R$ are either equal to the character $\chi$ of the representation over $\C$ (type 2) or equal to $\chi + \overline{\chi}$ or $2\chi$ (types 1 or 3). By averaging over the group, an invariant inner product can be created which allows each of these real representations to be written using real, orthogonal matrices. For many results, the orthogonality of the matrices is important. Therefore we will assume that all irreducibles appearing in the isotypic decompositions under consideration are of type 2. For $ S_n$ all irreducibles are of type 2, so this assumption is always justified. For other groups, to see if an irreducible representation is type 2, one needs check if $\frac{1}{|G|} \sum_{g \in G} \chi(g^2) = 1$ \cite[p. 109]{Serre1977linearRepresentationsOfFiniteGroups}. Whenever we use the complexification $\C \otimes_\R V$, recall that adjustments can be made so that all the matrices are real, and the dimensions will not change.

\begin{theorem} \label{thm:symmetry-adapted-matrix} \cite[Theorem 2.5]{FS1992}
Let $\rho: G \to GL(V)$ be a representation of the finite group $G$,
and let 
\begin{equation*}
    V = m_1 V_1 \oplus \cdots \oplus m_s V_s
\end{equation*}
be the direct sum decomposition into irreducible representations
$V_i$ with $\dim V_i=n_i$ and multiplicity $m_i$. Then every 
$P \in \text{Hom}(V,V)$ such that $D(g)P = PD(g)$ for all $g\in G$
has the following form in a symmetry adapted basis:
$$ P = \left( \begin{array}{cccc} 
P_1 & 0 & \dots & 0 \\
0  &  P_2 & \dots & 0\\
\vdots & \vdots & \ddots & \vdots \\
0  & 0 & \dots & P_s \end{array} \right)
$$
where each $P_i$ is an $(m_in_i) \times (m_in_i)$ matrix 
$$P_i = \left( \begin{array}{cccc}
\mu_{11}^i I_{n_i} & \mu_{12}^iI_{n_i} & \dots & \mu_{1m_i}^iI_{n_i} \\
\mu_{21}^iI_{n_i} & \mu_{22}^iI_{n_i} & \dots & \mu_{2m_i}^iI_{n_i} \\
\vdots & \vdots & \ddots & \vdots \\
\mu_{m_i1}^iI_{n_i} & \mu_{m_i2}^iI_{n_i} & \dots & \mu_{m_i m_i}^iI_{n_i} \end{array} \right).
$$ 
\end{theorem}
\begin{proof}
In a symmetry adapted basis we have
$$ D(g) = \left( \begin{array}{cccc}
D_1(g) & 0 & \dots & 0 \\
0 & D_2(g) & \dots & 0 \\
\vdots & \vdots & \ddots & \vdots \\
0 & 0 & \dots & D_s(g) \end{array} \right)
$$
where each $D_i(g)$ is an $(m_in_i) \times (m_in_i)$
block diagonal matrix with $m_i$ identical $n_i \times n_i$
matrices along its diagonal:
$$ D_i(g) = \left( \begin{array}{cccc}
\Sigma_i(g) & 0 & \dots & 0 \\
0  & \Sigma_i(g) & \dots & 0 \\
\vdots & \vdots & \ddots & \vdots \\
0 & 0 & \dots & \Sigma_i(g) \end{array} \right).
$$
After partitioning $P$ into 
$(m_in_i) \times (m_jn_j)$ matrices $P_{ij}$ for $i,j = 1, \ldots, s$,
we see that $D(g)P = PD(g)$ implies $D_i(g)P_{ij} = P_{ij}D_j(g)$. We partition each $P_{ij}$ further
$$
P_{ij} = \left( \begin{array}{cccc}
P_{ij}^{11} & P_{ij}^{12} & \dots & P_{ij}^{1m_j} \\
P_{ij}^{21} & P_{ij}^{22} & \dots & P_{ij}^{2m_j} \\
\vdots & \vdots & \ddots & \vdots \\
P_{ij}^{m_i1} & P_{ij}^{m_i2} & \dots & P_{ij}^{m_im_j} \end{array} \right)
$$
and observe that $\Sigma_i(g) P_{ij}^{tu} = P_{ij}^{tu} \Sigma_j(g)$
for all $i,j = 1, \ldots, s$ and $g \in G$. When we view $P_{ij}^{tu}$
as an element of $\text{Hom}(V_i, V_j)$, Schur's Lemma implies
that $P_{ij}^{tu} = 0$ whenever $i \neq j$. Furthermore, 
$P_{ii}^{tu} = \mu_{tu}^iI_{n_i}$.
\end{proof}

\begin{cor} \label{cor:general-dim} Let $V = m_1 V_1 \oplus \cdots \oplus m_s V_s$
be as in Theorem \ref{thm:symmetry-adapted-matrix}. Then the
dimension of the subspace of linear operators $P \in \text{Hom}(V,V)$
such that $D(g)P = PD(g)$ for all $g \in G$ is $m_1^2 + m_2^2 + \cdots + m_s^2$.
\end{cor}
\begin{proof}
The above theorem implies that the dimension is 
at most $m_1^2 + m_2^2 + \cdots + m_s^2$. Every block diagonal
matrix $P= \text{diag}(P_1, \ldots, P_s)$, with $P_i$
as in the theorem, commutes with each $D(g)$.
Since the $m_i^2$ scalars $\mu_{tu}^i$ are free parameters 
for $i=1, \ldots, s$, we get the result.
\end{proof}

A reordering of the symmetry adapted basis which block-diagonalizes the $D(g)$ matrices also leads to a more convenient block-diagonalization of commuting linear operators $P$ such that $P D(g) = D(g) P$.

\begin{cor} \label{cor:nice-diagonal}
Given $V = m_1 V_1 \oplus \cdots \oplus m_s V_s$ and $P \in \text{Hom}(V,V)$ such that $D(g)P=PD(g)$ for all $g\in G $, let 
$$\mathcal{B} = \bigcup_{i=1}^s \bigcup_{k=1}^{m_i} \mathcal{B}_{ik} $$
be an ordered basis that is symmetry adapted where 
 $\mathcal{B}_{ik} = \{v_1^{ik}, v_2^{ik}, \ldots, v_{n_i}^{ik}\}.$
If one reorders the basis vectors in $\bigcup_{k=1}^{m_i} \mathcal{B}_{ik}$
as $\bigcup_{\ell=1}^{n_i} \mathcal{\tilde{B}}_{i \ell}$
with 
$\mathcal{\tilde{B}}_{i \ell} = \{ v_\ell^{i1}, v_\ell^{i2}, \ldots, v_\ell^{i m_i}\}$
then 
$$ P = \left( \begin{array}{cccc} 
\tilde{P}_1 & 0 & \dots & 0 \\
0  &  \tilde{P}_2 & \dots & 0\\
\vdots & \vdots & \ddots & \vdots \\
0  & 0 & \dots & \tilde{P}_s \end{array} \right)
$$ 
where 
$$ \tilde{P}_i = \left( \begin{array}{cccc}
M_i & 0 & \dots & 0 \\
0 & M_i & \dots & 0 \\
\vdots & \vdots & \ddots & \vdots \\
0 & 0 & \dots & M_i \end{array} \right) \quad \quad \mbox{and} \quad \quad 
M_i = \left( \begin{array}{cccc}
\mu_{11}^i & \mu_{12}^i & \dots & \mu_{1 m_i}^i \\
\mu_{21}^i & \mu_{22}^i & \dots & \mu_{2 m_i}^i \\
\vdots & \vdots & \ddots & \vdots \\
\mu^i_{m_i 1} & \mu^i_{m_i 2} & \dots & \mu_{m_i m_i}^i \end{array} \right).
$$
\end{cor}
\begin{proof}
The reordering of the symmetry adapted basis has the effect 
of reordering the rows and columns of $P_i$ in Theorem \ref{thm:symmetry-adapted-matrix} resulting in $\tilde{P}_i$.
\end{proof}

For completeness, we briefly summarize the algorithm in \cite[p. 113]{FS1992} used to compute the change of basis matrix 
to get a symmetry adapted basis as in Corollary \ref{cor:nice-diagonal}. 
This algorithm can also be found in \cite[p. 23]{Serre1977linearRepresentationsOfFiniteGroups}. For each irreducible representation $V_i$ 
of the finite group $G$, let $d^i(g)$ be the matrix representation for 
$g \in G$.  The size of $d^i(g)$ is $n_i \times n_i$ where $n_i$ is the dimension of $V_i$. We furthermore choose $d^i(g)$ to be real orthogonal matrices, which can easily be done when all irreducibles appearing are of type 2, as we assume throughout.

\vskip 0.2cm
\begin{algorithm}[H]
$\bullet$ For each irreducible representation $i = 1,\ldots, s$,
\begin{enumerate}
    \item Compute the matrix
        $$\pi^{i} = \sum_{g \in G} d_{11}^i(g^{-1})D(g).$$
    \item The matrix $\pi^{i}$ will be of rank $m_i$.  Choose $m_i$ linearly independent columns and label them
        $$v_1^{i1}, v_1^{i2}, \ldots, v_1^{i m_i}.$$
        If this set of vectors is not orthonormal, apply Gram-Schmidt (here we utilize a modification to the algorithm \cite[Theorem 5.4]{FS1992}) and relabel each $v_1^{ij}$. 

    \item For each $k = 2, \ldots, n_i$,
        \begin{enumerate}
            \item Compute the matrix
            $$P_{ik} = \frac{n_i}{|G|}\sum_{g \in G} d_{1k}^i(g^{-1})D(g).$$
        \item Define new column vectors
            $$v_k^{ij} = P_{ik}v_1^{ij}$$
            for $j = 1, \ldots, m_i$.
        \end{enumerate}
\end{enumerate}        
$\bullet$ The above generates a symmetry adapted basis for all $m_i$ copies of $V_i$.  Arrange these vectors,
        $$\begin{array}{lcccc}
            \text{Basis } \mathcal{B}_{i1} \text{ for }V_i^1: & v_1^{i1} & v_2^{i1} & \cdots & v_{n_i}^{i1}  \\
            \text{Basis } \mathcal{B}_{i2} \text{ for }V_i^2: & v_1^{i2} & v_2^{i2} & \cdots & v_{n_i}^{i2}  \\
            \hspace{1.2cm} \vdots & \vdots & \vdots & \vdots & \vdots \\
            \text{Basis } \mathcal{B}_{i m_i} \text{ for }V_i^{m_i}: & v_1^{i m_i} & v_2^{i m_i} & \cdots & v_{n_i}^{i m_i}  \\
        \end{array}$$
        
$\bullet$ Construct the change of basis matrix $T$: For each $i = 1, \ldots, s$, the corresponding columns of $T$ will be the $\{v_k^{ij}\}$ in the following order:
            \begin{center}
                \begin{tikzpicture}
                  \matrix (m) [matrix of math nodes]
                  {
                    v_1^{i1}  & v_2^{i1} & \cdots & v_{n_i}^{i1}  \\
                    v_1^{i2}  & v_2^{i2} & \cdots & v_{n_i}^{i2}  \\
                    \vdots  & \vdots & \vdots & \vdots \\
                    v_1^{i m_i}  & v_2^{i m_i} & \cdots & v_{n_i}^{i m_i}  \\
                  };
                  \foreach \i in {1,2,4} \draw [->] (m-1-\i.north west) -- (m-4-\i.south west);
                  \draw [->] (m-4-1.south west) -- (m-1-2.north west);
                  \draw [->] (m-4-2.south west) -- (m-1-3.north west);
                \end{tikzpicture}
            \end{center}
            starting with $v_1^{i1}$ and going down each column of the array and ending with $v_{n_i}^{i m_i}$.
 \caption{Computation of symmetry adapted change of basis matrix as in Corollary~\ref{cor:nice-diagonal}}
 \label{alg:symmetry-adapted-basis}
\end{algorithm}

\subsection{Multiplicities of irreducible representations for $S_n$ acting on homogeneous polynomials}\label{subsection:combinatorial-rule-for-Sn-multiplicities}

In this paper, $V = \R[x_1,\dots,x_n]_d \simeq \R^N$ or its complexification, the vector space of homogeneous degree $d$ polynomials in $x_1,\dots,x_n$. In this section we begin with a representation of $G = S_n$ on $\R^n$ which extends to a representation on $V$ by linear substitution of variables. 
Furthermore, 
we will also use the fact that the irreducible representations 
of $S_n$ are indexed by partitions $\lambda$ \cite{Sagan}. In other words, 
\begin{equation*}
    V = m_{\lambda_1} V_{\lambda_1} \oplus \cdots \oplus m_{\lambda_s} V_{\lambda_s}
\end{equation*}
where $\lambda_1, \lambda_2, \ldots, \lambda_s$ are partitions of $n$.
Here we provide a simple way to determine the multiplicity
of the irreducible representation $V_\lambda$.
For this we need to compute $\langle \chi_\lambda, \, \chi_d \rangle$
where $\chi_\lambda$ is the irreducible character associated
to $V_\lambda$ and $\chi_d$ is the character of the representation
$V = \C[x_1,\ldots,x_n]_d$. We will present a method which we have learned from Mark Haiman.

Recall that the space of complex-valued functions $\C^G$ on a group has a natural inner product $\C^G \times \C^G \to \C$ defined by
\begin{equation*}
    \langle f,g \rangle := \frac{1}{|G|} \sum_{\sigma \in G} \overline{f(\sigma)} g(\sigma).
\end{equation*}
The ring of symmetric functions $\Lambda$ also has a natural inner product. This can be defined by specifying its values on pairs of basis vectors; for instance
\begin{equation*}
    \langle m_\lambda, h_\mu \rangle = \delta_{\lambda \mu}
\end{equation*}
where $m_\lambda$ and $h_\mu$ are monomial and complete homogeneous
symmetric functions associated to partitions $\lambda$ and $\mu$, respectively. Elsewhere in this paper $m_\lambda$ denotes the multiplicity of the irreducible representation $V_\lambda$, but in this section it denotes the monomial symmetric function associated to such a partition.
A key tool for us will be the \textit{Frobenius characteristic map} \cite[p. 351]{Stanley2}. This is a linear map between the subspace of functions $\chi:S_n \to \mathbb{C}$ constant on conjugacy classes and the ring $\Lambda$. It is defined by 
$$\mathrm{ch}(\chi) = \frac{1}{n!}\sum_{\sigma \in S_n} \chi(\sigma)p_{\mathrm{par}(\sigma)}$$
where $\mathrm{par}(\sigma) = \mu$ is the partition
given by the cycle type of $\sigma$, 
and $p_\mu = p_{\mu_1} \cdots p_{\mu_k}$ is the power sum symmetric polynomial \cite[Section 7.7]{Stanley2}. The characteristic map $\mathrm{ch}$ is an isometry \cite[Proposition 7.18.1]{Stanley2} between the subspace of functions constant on conjugacy classes and the space $\Lambda_n$ of degree $n$ symmetric functions, each equipped with their respective inner products. In the former, the irreducible characters $\chi_\lambda$ of $S_n$ form an orthonormal basis, and in the latter, the Schur polynomials $s_\lambda$ form an orthonormal basis. It is a standard fact in representation theory and the theory of symmetric functions that $\mathrm{ch}(\chi_\lambda) = s_\lambda$ 
\cite[Section 4.7]{Sagan}.

\begin{theorem}\label{theorem:multiplicity-as-integer-solutions}
Let $\chi_d$ be the character of the representation of the symmetric group $S_n$ acting on polynomials of degree $d$ in $n$ variables $V = \mathbb{C}[x_1,\dots,x_n]_d$. Let $n(\lambda) = \sum_i (i-1) \lambda_i$ and let $h_i$ be the hook length for the $i$th box in the Young diagram of $\lambda$. The multiplicity of the irreducible representation $V_\lambda$ in $V$ is equal to the number of solutions $y \in \mathbb{N}^n$ of the equation
\begin{equation*}
    h_1 y_1 + \cdots + h_n y_n = d - n(\lambda).
\end{equation*}
\end{theorem}

\begin{proof}
We first compute the inner product
\begin{align*}
    \langle s_\lambda(z), \sum_d \mathrm{ch}(\chi_d) q^d \rangle &= \langle s_\lambda(z), \sum_{\mu \vdash n} s_\mu(z)s_\mu(1,q,q^2, \ldots) \rangle\\
&= s_\lambda(1,q,q^2, \ldots).
\end{align*}
Here, the first equality  is by \cite[Exercise 7.73]{Stanley2}, while the second one is by orthonormality of the Schur basis for $\Lambda$. Thus we have shown that 
$$ \sum_d \langle \chi_\lambda,\chi_d \rangle q^d 
= \sum_d \langle \mathrm{ch}(\chi_\lambda), \mathrm{ch}(\chi_d)\rangle q^d
= \langle s_\lambda(z), \sum_d \mathrm{ch}(\chi_d) q^d \rangle
= s_\lambda(1,q,q^2, \ldots).
$$
Let $f_\lambda(q)$ be the $q$-analogue of the number of standard Young tableaux of shape $\lambda$, which means that $$f_\lambda(q) = \sum_{T\in SYT(\lambda)}q^{\mathrm{maj}(T)}$$ where $\mathrm{maj}(T)$ is the sum of the descents in $T$, i.e. it is the sum over all $i$ such that $i+1$ appears in a lower row in $T$ than $i$. We let $h(x)$ be the hook length for a box $x$ in the Young diagram of $\lambda$. Using this, we obtain
$$
s_\lambda(1,q,q^2, \ldots)
= \frac{f_\lambda(q)}{(1-q)(1-q^2)\cdots (1-q^n)}
= \frac{q^{n(\lambda)}}{\prod_{x\in \lambda}(1-q^{h(x)})}$$
$$= q^{n(\lambda)}(1 + q^{h_1} + q^{2h_1} + \cdots)(1 + q^{h_2} + q^{2h_2} + \cdots) \cdots (1 + q^{h_n} + q^{2h_n} + \cdots) $$
where the first equality is \cite[Proposition 7.19.11]{Stanley2}, 
the second equality is \cite[Corollary 7.21.3]{Stanley2}, and
$h_i$ are all the hook lengths of $\lambda$. Expanding this out we see that the coefficient of the $q^d$ term is the number of ways we can add multiples of the hook lengths $h_1, 2h_1, \ldots, h_2, 2h_2, \ldots, h_i, 2h_i, \ldots$ so they add up to $d - n(\lambda)$.
\end{proof}

\subsection{Sum of squares and Gram spectrahedra}\label{subsection:SOSs-and-Gram-spectrahedra}

A homogeneous polynomial $f$ in $\R[x_1,\ldots, x_n]_{2d}$ is said 
to be a \emph{sum of squares} (SOS) polynomial if 
$f = q_1^2 + \cdots + q_k^2$ where $q_i \in \R[x_1,\ldots,x_n]_d$, 
$i=1,\ldots, k$. 
Note that while we are searching for an SOS decomposition of a degree $2d$ polynomial, most of the work occurs in the space of degree $d$ polynomials, including the use of representation theory. The following is a well known fact that drives many ideas in the theory and practice of SOS polynomials; see for instance \cite[Theorem 3.39]{Parrilo13}.
\begin{theorem} \label{thm:how-to-write-f}
Let $f(x) \in \R[x_1, \ldots, x_n]_{2d}$ be a homogeneous polynomial and let $m(x)$ be 
a column vector containing a basis of $\R[x_1, \ldots, x_n]_d$.
Then $f(x)$ is a sum of squares if and only if 
there exists an $N \times N$ real positive semidefinite symmetric matrix $Q$ where 
$N = \binom{n+d-1}{d}$ and 
\begin{equation}\label{equation:sos-representation-by-monomial-gram-monomial}
    f(x) = m(x)^T Q m(x).
\end{equation}
\end{theorem}

The set of $N\times N$ real symmetric matrices $\Sym^N$ is a vector space isomorphic to $\R^{\binom{N+1}{2}}$. The subset of positive semidefinite matrices $\PSD_N$ is a full-dimensional closed convex cone in this vector
space. It is a semi-algebraic set defined by $2^N - 1$ polynomial inequalities given by forcing the $2^N - 1$ principal minors of an $N \times N$ symmetric matrix to be nonnegative. 

A \textit{spectrahedron} $K$ is a closed convex semi-algebraic set, formed as the intersection of some affine linear space $L \subset \Sym^N$ with $\PSD_N$. Spectrahedra are generalizations of polyhedra, which
are feasible sets of linear programming problems. Similarly, spectrahedra are the feasible sets of \textit{semidefinite programming problems} (SDP):
\begin{equation*}
    \min \quad \langle C, Q \rangle \quad \mbox{such that} \quad Q \in K
\end{equation*}
where $\langle C, Q \rangle := \text{trace}(C^TQ) = \sum_{i=1}^{N} \sum_{j=1}^N C_{ij} Q_{ij}$
is the standard inner product on $\Sym^N$.
SDPs can be solved efficiently. In particular,
whether a spectrahedron is empty or not can be decided by using the dual SDP problem \cite{WSV2000}.

\begin{definition}\label{definition:gram-spectrahedron}
Let $f \in \R[x_1, \ldots, x_n]_{2d}$. The \textit{Gram spectrahedron} of $f$ is the spectrahedron
\begin{equation*}
    K_f := L_f \cap \PSD_N,
\end{equation*}
where $L_f$ is the affine subspace of symmetric matrices $Q$ satisfying (\ref{equation:sos-representation-by-monomial-gram-monomial}).
\end{definition}
\begin{prop}
The Gram spectrahedron $K_f$ is non-empty if and only if $f$ is an SOS polynomial.
\end{prop}
In other words, determining if a polynomial is SOS is equivalent to checking the feasibility of an SDP. Gram spectrahedra have been
studied intensively in \cite{BGPS-Low-rank-SOS, ChoiLamReznick1995sumsofsquaresofrealpolys, GramSpectrahedra2017chuaplaumannsinnvinzant, GP2004, PSV2011}, to name a few.

\subsection{Symmetry adapted Gram spectrahedra}

This article focuses on SOS polynomials invariant under the linear action of a group $G$. Therefore we start with a representation of $G$ on $\mathbb{R}^n$. A polynomial $f$ is $G$-invariant if $f(g^{-1}x) = f(x)$ for all $g \in G$. The $S_n$-invariant polynomials are the usual symmetric polynomials. The ring of $G$-invariant polynomials of degree $2d$ will be denoted $\mathbb{R}[x_1,\dots,x_n]_{2d}^G$.

The action of $G$ on $\mathbb{R}^n$ extends to a representation $D:G \to GL(V)$ for $V = \mathbb{R}[x_1,\dots,x_n]_d$ with matrices $D(g)$ with respect to a chosen basis. Let $m(x)$ be the column vector whose entries form a basis for $V$. For \textit{any} (possibly non-invariant) polynomial $f \in \mathbb{R}[x_1,\dots,x_n]_{2d}$ we can write $f(x) = m(x)^T Q m(x)$ for some $Q \in \Sym^N$. Hence $g \cdot f(x) = m(x)^T D(g)^T Q D(g) m(x)$ for all $g \in G$, and if $f$ is $G$-invariant then $f(x) = m(x)^T D(g)^T Q D(g) m(x)$ for all $g \in G$.

\begin{prop} \label{prop:existence-of-symmetry-adapted-Q}
If $f$ is a $G$-invariant  polynomial in $\R[x_1,\ldots,x_n]_{2d}^G$ then there exists $Q\in \Sym^N$
such that $f(x) = m(x)^T Q m(x)$ where $Q = D(g)^T Q D(g)$ for all $g \in G$.
\end{prop}
\begin{proof}
By Theorem \ref{thm:how-to-write-f} there exists $Q' \in \Sym^N$ such that $f(x) = m(x)^T Q' m(x)$.  Since $f$ is $G$-invariant, 
$f(x) = m(x)^T D(g)^T Q' D(g) m(x)$ for all $g \in G$. Now
let 
$$
Q = \frac{1}{|G|} \sum_{g \in G} D(g)^T Q' D(g).
$$
\end{proof}

\begin{definition}\label{definition:NEW-symm-adapted-gram-spectrahedron}
Let $f \in \R[x_1,\ldots, x_n]_{2d}^G$ be a $G$-invariant polynomial for some representation of $G$ on $\mathbb{R}^n$. Let $D:G \to GL(V)$ be the representation of $G$ on $V = \mathbb{R}[x_1,\dots,x_n]_d$ given by linear substitution. The \textit{symmetry adapted Gram spectrahedron} of $f$ is 
\begin{equation*}
    K_f^G := L_f \cap \PSD_N^G,
\end{equation*}
where
\begin{equation*}
    \PSD_N^G := \bigg\{ Q \in \PSD_N \hspace{0.3cm} \bigg\rvert \hspace{0.3cm} D(g)^T Q D(g) = Q, \text{ for all }g \in G \bigg\}.
\end{equation*}
Here, $L_f$ is the affine space of symmetric matrices $Q$ satisfying $f(x) = m(x)^T Q m(x)$ for $m(x)$ a column vector whose entries form a basis of $V$, and $D(g)$ are the matrices of $D$ in this basis. The set $\PSD_N^G$ consists of all positive semidefinite matrices which are fixed by the action of $G$. We call this the \textit{symmetry adapted PSD cone}.
\end{definition}

\begin{cor} \label{cor:characterize-psd-cone}
Let $V = \R[x_1,\ldots, x_n]_d$ and let $D: G \to GL(V)$ be the representation of $G$ on $V$ obtained by linear substitution from a representation of $G$ on $\R^n$. Assume that all irreducible representations appearing in the isotypic decomposition
\begin{equation*}
    \mathbb{C} \otimes_\mathbb{R} V = m_1 V_1 \oplus \cdots \oplus m_s V_s
\end{equation*}
are of type 2, with $\dim V_i =n_i$ and multiplicity $m_i$. Then there exists a basis for $V$ such that a symmetric matrix $Q \in \Sym^N$ is in $\PSD_N^G$ if and only if 
\begin{equation} \label{eq:nice-form}
Q = \left( \begin{array}{cccc} 
\tilde{Q}_1 & 0 & \dots & 0 \\
0  &  \tilde{Q}_2 & \dots & 0\\
\vdots & \vdots & \ddots & \vdots \\
0  & 0 & \dots & \tilde{Q}_s \end{array} \right)
\quad \quad
\mbox{where}
\quad \quad
 \tilde{Q}_i = \left( \begin{array}{cccc}
Q_i & 0 & \dots & 0 \\
0 & Q_i & \dots & 0 \\
\vdots & \vdots & \ddots & \vdots \\
0 & 0 & \dots & Q_i \end{array} \right)
\end{equation}
with $Q_i \in \PSD_{m_i}$ for all $i=1,\ldots,s$ and $n_i$ identical
copies in $Q_i$.
\end{cor}

\begin{proof}
By Corollary \ref{cor:nice-diagonal} an arbitrary matrix $Q$ commutes with all $D(g)$ if and only if it has the stated block-diagonal form in a symmetry adapted basis. If all matrices $D(g)$ are orthogonal, requiring $D(g)^T Q D(g) = Q$ is the same as requiring $Q D(g) = D(g) Q$. Since the irreducibles are of type 2, the matrices $d^i$, and therefore $\pi^i$ and $P_{ik}$, can be chosen with real entries in Algorithm \ref{alg:symmetry-adapted-basis}. Thus, the symmetry adapted basis can be written as real linear combinations of the original basis vectors. By using the invariant inner product
\begin{equation*}
    \langle v, w \rangle := v^T \left( \frac{1}{|G|}\sum_{g \in G} D(g)^T D(g) \right) w,
\end{equation*}
the symmetry adapted basis can further be adjusted so that the matrices $D(g)$ in that basis are orthogonal matrices. To carry this out, one can apply Gram-Schmidt using the invariant inner product above. It only remains to require symmetry and positive semi-definiteness. This is the condition stated above, that $Q_i \in \PSD_{m_i}$.
\end{proof}

As one might expect $\PSD_N^G$ and $K_f^G$ are simpler, smaller, and more structured objects than their counterparts $\PSD_N$ and $K_f$ when $f$ is $G$-invariant. The rest of this article is devoted to convincing the reader that this is indeed the case.

\begin{ex} \label{ex:icosahedral} 
The focus of this paper is the case $G = S_n$. However, we include an example with the symmetry group $G = I_h$ of an icosahedron. 
All 10 irreducible representations of $I_h$ are of type 2. We continue this example in Section \ref{subsection:extreme-rays} to demonstrate extremal rays of rank $>1$. 
This group consists of $120$ invertible $3 \times 3$ orthogonal matrices. Generators are, for instance,
\begin{equation*}
    \left(\begin{array}{rrr}
-1 & 0 & 0 \\
0 & -1 & 0 \\
0 & 0 & 1
\end{array}\right) , \left(\begin{array}{rrr}
0 & 0 & 1 \\
1 & 0 & 0 \\
0 & 1 & 0
\end{array}\right) , \left(\begin{array}{rrr}
\frac{1}{2} & -\frac{1}{4} \, \sqrt{5} - \frac{1}{4} & \frac{1}{\sqrt{5} + 1} \\
\frac{1}{4} \, \sqrt{5} + \frac{1}{4} & \frac{1}{\sqrt{5} + 1} & -\frac{1}{2} \\
\frac{1}{\sqrt{5} + 1} & \frac{1}{2} & \frac{1}{4} \, \sqrt{5} + \frac{1}{4}
\end{array}\right) , \left(\begin{array}{rrr}
-1 & 0 & 0 \\
0 & -1 & 0 \\
0 & 0 & -1
\end{array}\right).
\end{equation*}
The action on $\mathbb{R}^3$ extends to an action on 
$V = \R[x_1,x_2,x_3]_2$. The $6 \times 6$ matrices $\widetilde{D(g)}$ for all $120$ elements $g \in I_h$ written in the monomial basis $\{x_1^2, x_1x_2, x_1x_3, x_2^2, x_2x_3, x_3^2\}$ are
\begin{equation*} \tiny
    \left(\begin{array}{rrr}
g_{11} & g_{12} & g_{13} \\
g_{21} & g_{22} & g_{23} \\
g_{31} & g_{32} & g_{33}
\end{array}\right) \xmapsto{\widetilde{D}} \left(\begin{array}{rrrrrr}
g_{11}^{2} & g_{11} g_{21} & g_{11} g_{31} & g_{21}^{2} & g_{21} g_{31} & g_{31}^{2} \\
2 \, g_{11} g_{12} & g_{12} g_{21} + g_{11} g_{22} & g_{12} g_{31} + g_{11} g_{32} & 2 \, g_{21} g_{22} & g_{22} g_{31} + g_{21} g_{32} & 2 \, g_{31} g_{32} \\
2 \, g_{11} g_{13} & g_{13} g_{21} + g_{11} g_{23} & g_{13} g_{31} + g_{11} g_{33} & 2 \, g_{21} g_{23} & g_{23} g_{31} + g_{21} g_{33} & 2 \, g_{31} g_{33} \\
g_{12}^{2} & g_{12} g_{22} & g_{12} g_{32} & g_{22}^{2} & g_{22} g_{32} & g_{32}^{2} \\
2 \, g_{12} g_{13} & g_{13} g_{22} + g_{12} g_{23} & g_{13} g_{32} + g_{12} g_{33} & 2 \, g_{22} g_{23} & g_{23} g_{32} + g_{22} g_{33} & 2 \, g_{32} g_{33} \\
g_{13}^{2} & g_{13} g_{23} & g_{13} g_{33} & g_{23}^{2} & g_{23} g_{33} & g_{33}^{2}
\end{array}\right).
\end{equation*}
The resulting $6 \times 6$ matrices above will not be orthogonal matrices. However, we can create the matrix
\begin{equation*}
    S := \frac{1}{|G|} \sum_{g \in G} \widetilde{D(g)}^T \widetilde{D(g)}
\end{equation*}
which we use to define the invariant inner product $\langle v,w \rangle := v^T S w$. In this case,
\begin{equation*}
    S = \left(\begin{array}{rrrrrr}
\frac{7}{5} & 0 & 0 & -\frac{1}{5} & 0 & -\frac{1}{5} \\
0 & \frac{4}{5} & 0 & 0 & 0 & 0 \\
0 & 0 & \frac{4}{5} & 0 & 0 & 0 \\
-\frac{1}{5} & 0 & 0 & \frac{7}{5} & 0 & -\frac{1}{5} \\
0 & 0 & 0 & 0 & \frac{4}{5} & 0 \\
-\frac{1}{5} & 0 & 0 & -\frac{1}{5} & 0 & \frac{7}{5}
\end{array}\right).
\end{equation*}
Applying a modified Gram-Schmidt to the monomial basis we can create a new basis $u_1,\dots,u_6$ for which the representation matrices become orthogonal. Collecting the new basis vectors in the columns of a matrix $U$ we create orthogonal matrices $D(g) = U^{-1} \widetilde{D(g)} U$ for all $g \in I_h$. A useful fact is that $U^{-1} = U^T S$. Consider $m(x)^T I m(x)$ for $m(x)$ the column vector containing the monomials of degree 2. This would produce the polynomial
\begin{equation*}
    x_1^4 + x_1^2 x_2^2 + x_1^2 x_3^2 + x_2^4 + x_2^2 x_3^2 + x_3^4,
\end{equation*}
which is not $I_h$-invariant. Proposition \ref{proposition:invariant-isotypic-partial-sums} below implies that, in the monomial basis, the identity matrix is not in $PSD_N^{I_h}$, as can also be checked directly. However, if we apply the change of basis and extract the polynomial corresponding to the identity matrix $f = (U^T m)^T I (U^T m)$ we obtain the $I_h$-invariant polynomial
\begin{equation*}
    f = \frac{3}{4} \, x_{1}^{4} + \frac{3}{2} \, x_{1}^{2} x_{2}^{2} + \frac{3}{4} \, x_{2}^{4} + \frac{3}{2} \, x_{1}^{2} x_{3}^{2} + \frac{3}{2} \, x_{2}^{2} x_{3}^{2} + \frac{3}{4} \, x_{3}^{4}.
\end{equation*}
In the basis given by the column vector $U^T m$, the 2-dimensional symmetry adapted PSD cone $\PSD_6^{I_h}$ is given by the 63 inequalities arising from the principal minors of the matrix given (to 5 digits) by
\begin{equation*}\footnotesize
    \left(\begin{array}{rrrrrr}
\frac{13}{28} \, q_{55} + \frac{15}{28} \, q_{66} & 0 & 0 & -0.61859 \, q_{55} + 0.61859 \, q_{66} & 0 & -0.73193 \, q_{55} + 0.73193 \, q_{66} \\
0 & q_{55} & 0 & 0 & 0 & 0 \\
0 & 0 & q_{55} & 0 & 0 & 0 \\
-0.61859 \, q_{55} + 0.61859 \, q_{66} & 0 & 0 & \frac{2}{7} \, q_{55} + \frac{5}{7} \, q_{66} & 0 & -0.84515 \, q_{55} + 0.84515 \, q_{66} \\
0 & 0 & 0 & 0 & q_{55} & 0 \\
-0.73193 \, q_{55} + 0.73193 \, q_{66} & 0 & 0 & -0.84515 \, q_{55} + 0.84515 \, q_{66} & 0 & q_{66}
\end{array}\right).
\end{equation*}

\end{ex}

We close this section with the observation that constructing SOS decompositions with symmetry adapted bases has another advantage. Namely, the partial sums of squares are $G$-invariant polynomials themselves when one groups them according to the isotopic components. This result was also pointed out in \cite[pp. 107-112]{GP2004}, but we would like to call attention to it, as well as provide a fully explicit proof. We then use this result to prove that every matrix in $PSD_N^G$ produces a $G$-invariant polynomial.

\begin{prop}\label{proposition:invariant-isotypic-partial-sums}
Let $f \in \mathbb{R}[x_1,\dots,x_n]_{2d}^G$ be a $G$-invariant polynomial with real coefficients  and let every irreducible appearing with nonzero multiplicity in $\mathbb{C}[x_1,\dots,x_n]_d = \mathbb{C} \otimes_\mathbb{R} \mathbb{R}[x_1,\dots,x_n]$ be of type 2.  If $f$ is an SOS polynomial then
\begin{equation}
    f = \sum_{\alpha_1 =1 }^{r_1} q_{1,\alpha_1}^2 + \sum_{\alpha_2 = 1}^{r_2} q_{2,\alpha_2}^2 + \cdots + \sum_{\alpha_s = 1}^{r_s} q_{s,\alpha_s}^2
\end{equation}
where each $q_{i,\alpha_i}$  is a polynomial of degree $d$ appearing in the $i$th isotypic component $m_i V_i$ of
\begin{equation*}
    \C[x_1,\ldots,x_n]_d = m_1 V_1 \oplus \cdots \oplus m_s V_s.
\end{equation*}
Further, each partial sum of squares  $\sum_{\alpha_i=1}^{r_i} q_{i,\alpha_i}^2$
is a $G$-invariant polynomial, with $r_i = \mathrm{rank}(Q_i)$ as in Corollary \ref{cor:characterize-psd-cone}. By choosing bases agreeing with the real representations corresponding to each isotypic component, each $q_{i,\alpha_i}$ may be chosen with real coefficients.
\end{prop}

\begin{proof}
Let $v^i_j$ be the column vector $[v^{i1}_j, v^{i2}_j, \dots, v^{im_i}_j]^T$ of basis polynomials chosen in Algorithm \ref{alg:symmetry-adapted-basis} as an orthonormal basis for the column space of the $j$th projection operator for the $i$th isotypic component. Since $V_i$ is of type 2, these basis vectors can be chosen as polynomials with real coefficients, and such that the matrices $d^i(g)$ are orthogonal. Let $Q_i$ be the matrices appearing in Corollary \ref{cor:characterize-psd-cone}. Then the partial sum of squares for the $i$th isotypic component can be rewritten
\begin{align*}
    \sum_{\alpha_i =1}^{r_i} q_{i,\alpha_i}^2 &= \sum_{j=1}^{n_i} (v^i_j)^T Q_i (v^i_j)\\
     &= \left\langle Q_i, \sum_{j=1}^{n_i} (v^i_j) (v^i_j)^T \right\rangle\\
     &= \left\langle Q_i, P_i(x) \right\rangle
\end{align*}
where $P_i(x)$ is an $m_i \times m_i$ matrix with polynomial entries
and $r_i = \mathrm{rank}(Q_i)$. Specifically, the $(k,\ell)$ entry of the matrix $P_i(x)$ is given by
\begin{equation}\label{equation:Pi-entries}
    p^i_{k,\ell} = \sum_{j=1}^{n_i} v^{ik}_j v^{i\ell}_j.
\end{equation}
Letting $d^i(g) = (d_{\alpha,\beta}^i)$ for $g \in G$ be the orthogonal matrices for the real representation associated to the $i$th isotypic component, we have the relations
\begin{equation*}
    \sum_{j = 1}^{n_i} (d_{\alpha j}^i) (d_{\beta j}^i) = \delta_{\alpha \beta}.
\end{equation*}
Recall for each $k=1, \ldots, m_i$  the entry $v^{ik}_j$ of the column vector $v^i_j$ is a symmetry adapted basis polynomial which \textit{transforms} like the $j$th basis vector of the $i$th irreducible representation:
\begin{equation*}
     g \cdot v^{ik}_j  \, = \,  \sum_{\alpha = 1}^{n_i} d_{\alpha j}^i v^{ik}_\alpha.
\end{equation*}
Acting with the group element $g \in G$ we have
\begin{align*}
    \sum_{j=1}^{n_i} (v^i_j) (v^i_j)^T \mapsto &\sum_{j = 1}^{n_i} \left[ \begin{array}{c}
        \vdots \\
        \sum_{\alpha=1}^{n_i} d_{\alpha j}^i v^{ik}_\alpha \\
        \vdots
    \end{array} \right] \left[ \begin{array}{ccc}
        \cdots & \sum_{\beta=1}^{n_i} d_{\beta j}^i v^{i \ell}_\beta & \cdots
    \end{array} \right] \\
     & = \sum_{j=1}^{n_i} \left[ \begin{array}{ccc}
         \ddots & (k,\ell)\text{ entry}  = & \\
          & \sum_{(\alpha,\beta) \in [n_i]\times [n_i]} d_{\alpha j}^i d_{\beta j}^i v^{ik}_\alpha v^{i \ell}_\beta & \\
          & & \ddots
     \end{array} \right].
\end{align*}
Pulling the sum over $j=1,\ldots, n_i$ inside to each individual entry of the matrix we see that the orthogonality relations zero out all terms except those giving the $(k,\ell)$ entry of $P_i(x)$. Therefore, each of the entries of $P_i(x)$ will be itself an invariant polynomial, and
hence 
$\sum_{\alpha_i=1}^{r_i} q_{i,\alpha_i}^2 = \langle Q_i, P_i(x)\rangle$
is invariant. Note that a factorization of $Q_i$ will still be required to find the $r_i$ explicit squares $q_{i,\alpha_i}^2$ as usual.
\end{proof}

\begin{ex}\label{example:H21-H111-partial-sums-invariant}
Consider the polynomial $(H_{21} - H_{111})(x_1^2,x_2^2,x_3^2) = $
\begin{align*} 
	& \frac{1}{18}(x_1^4 + x_1^2x_2^2 + x_1^2x_3^2 + x_2^4 + x_2^2x_3^2 + x_3^4)(x_1^2 + x_2^2 + x_3^2) - \frac{1}{27}(x_1^2 + x_2^2 + x_3^2)^3\\
	&= \frac{1}{54}x_1^6 + \frac{1}{54}x_2^6 + \frac{1}{54}x_3^6 - \frac{1}{18} x_1^2x_2^2x_3^2 
\end{align*}
which is an $S_3$-invariant (symmetric) polynomial. We will define 
a family of such polynomials in Section \ref{sec:inequalities}.
One matrix in its symmetry adapted Gram spectrahedron is
$$\frac{1}{108}\begin{pmatrix}
4 & -2\sqrt{2} & 0 & 0 & 0& 0& 0& 0& 0 &0\\
-2\sqrt{2} & 2 & 0& 0 & 0& 0& 0& 0& 0& 0\\
 0 & 0 & 0 & 0 & 0& 0& 0& 0& 0 & 0\\
 0 & 0 & 0 & 1 & \sqrt{2} & 0  & 0 & 0 & 0 & 0 \\
 0 & 0 & 0 &  \sqrt{2}  & 2 & 	0 & 0 & 0& 0& 0 \\
 0 & 0 & 0 &   0 & 0 &  0 & 0 & 0& 0& 0\\
 0 & 0 & 0 & 0  & 0 & 0& 1 & \sqrt{2}   & 0 & 0 \\
 0 & 0 & 0 & 0 & 0 & 0& \sqrt{2}  & 2 & 	 0& 0 \\
 0 & 0 & 0 &   0 & 0 &  0 & 0 & 0& 0& 0\\
 0 & 0 & 0 & 0 & 0 & 0 & 0 & 0 & 0 & 6
\end{pmatrix}. $$
The rows and columns of this matrix correspond to polynomials which
form a symmetry adapted basis, and using these we can write our polynomial as 
\begin{align*} 
	(H_{21} - H_{111})&(x_1^2,x_2^2,x_3^2) =  \frac{1}{54}x_1^6 + \frac{1}{54}x_2^6 + \frac{1}{54}x_3^6 - \frac{1}{18} x_1^2x_2^2x_3^2 \\
	&= \frac{1}{108}\Big(\big(\frac{2\sqrt{3}}{3}(x_1^3 + x_2^3 + x_3^3)
			-\frac{\sqrt{3}}{3}(x_1^2x_2 + x_1^2x_3 + x_1x_2^2 + x_1x_3^2 + x_2^2x_3 + x_2x_3^2)\big)^2\\
	& + \big(\frac{\sqrt{6}}{6}(2x_1^3 - x_2^3 - x_3^3) 
			+ \frac{\sqrt{6}}{6}(2x_1^2x_2 + 2x_1^2x_3 - x_1x_2^2 - x_1x_3^2 - x_2^2x_3 - x_2x_3^2)\big)^2\\
	& + \big(\frac{\sqrt{2}}{2}(x_2^3 - x_3^3 )
			+ \frac{\sqrt{2}}{2}(x_1x_2^2 - x_1x_3^2 + x_2^2x_3 - x_2x_3^2)\big)^2\\
	& + \big(x_1^2x_2 - x_1^2x_3 - x_1x_2^2 + x_1x_3^2 + x_2^2x_3 - x_2x_3^2\big)^2\Big)
\end{align*}
where the first square comes from the rank one trivial block, the second and third squares from the two copies of the rank one standard block and the last square from the rank one alternating block.  Clearly, the first
and last squares are symmetric polynomials. Proposition \ref{proposition:invariant-isotypic-partial-sums} states that the sum of the second and third squares is also a symmetric polynomial. Although 
it is not immediately clear from the above representation, it is indeed so. We invite the reader to check. 
\end{ex}

Note that the proof of Proposition \ref{proposition:invariant-isotypic-partial-sums} can be applied to any matrix $Q$ in the symmetry adapted PSD cone, which leads to the following results.

\begin{cor}
Let $m(x)$ be a vector of polynomials comprising a fixed basis of $\R[x_1,\dots,x_n]_d$. Then every matrix $Q \in \PSD_N^G$, calculated using the representation matrices $D(g)$ written in this basis,  produces a $G$-invariant polynomial  $f(x) = m(x)^T Q m(x)$.
\end{cor}

\begin{cor}\label{corollary:Kf-nonempty}
The symmetry adapted Gram spectrahedron $K_f^G$ is non-empty if and only if the $G$-invariant polynomial $f$ is SOS.
\end{cor}

\section{Properties of symmetry adapted PSD cones and Gram spectrahedra} \label{section:properties-of-symm-adapted-PSD-cone}

In this section we provide general results about $\PSD_N^G$ and $K_f^G$. We compute the dimension of $\PSD_N^G$ and give a characterization of its extreme rays, as well as describe the matrix block of $Q \in \PSD_N^{S_n}$ in a symmetry adapted basis corresponding to the trivial representation when $G$ is the symmetric group.

\begin{cor}\label{dimcor}
The dimension of $\PSD_N^G$ is $\sum_{i=1}^s \binom{m_i + 1}{2}$.
\end{cor}
\begin{proof}
Since the dimension of $\PSD_{m_i}$ is $\binom{m_i + 1}{2}$ Corollary
\ref{cor:characterize-psd-cone} implies the result.
\end{proof}

\subsection{Extremal Rays}\label{subsection:extreme-rays}

Every point $Q \in PSD_N^G$ gives rise to a \textit{ray}, as in
\begin{equation*}
    \text{ray}(Q) := \left\{ c Q : c \in \mathbb{R}_{\geq 0} \right\}.
\end{equation*}
A ray $r$ is \textit{extremal} if it cannot be written as a non-trivial convex combination of other rays. We note that in the case
of the usual cone of positive semidefinite matrices $\PSD_N$,
the Spectral Theorem for symmetric matrices implies that 
the extremal rays correspond to matrices of rank one. 
A \textit{face} $F$ of a convex set $K$ is a convex subset such that if a convex combination of two points of $K$ lies in $F$, then the points were already elements of $F$. In symbols, if $a,b \in K$ and $ta +(1-t)b \in F$ for some $t \in (0,1)$ then $a,b \in F$. 
Given a spectrahedron $K$, any matrix $Q \in K$ belongs to the relative interior of a unique face denoted by $F_K(Q)$.  The face $F_K(Q)$ is the intersection of $K$ with the subspace of all matrices whose kernel contains the kernel of $Q$; see \cite{GR1995}. 

\begin{theorem} \cite[Theorem 1]{GR1995} \label{UnqFaceThm}
Let $K \subset \PSD_k$ be a spectrahedron, and for $Q \in K$ define
$$ S(Q) = \{X \in \Sym^k \, : \, \ker(Q) \subset  \ker(X) \}.$$
Then $F_K(Q) = S(Q) \cap K$.
\end{theorem}

\begin{cor}
Let $K = L \cap PSD_k$ be a spectrahedral cone for some linear subspace $L \subset \Sym^k$, and let $Q\in K$.
Then $Q$ is extremal if and only if the dimension of the affine hull
of $F_K(Q)$ is one.
\end{cor}
This leads to the following theorem, further specialized to our case: 

\begin{theorem}\label{theorem:ranks-and-extreme-rays}
Let $Q_1,\dots,Q_s$ be the symmetric matrices appearing in the blocks as in Corollary~\ref{cor:characterize-psd-cone}. Then the extremal rays of $\PSD_N^G$ are in bijection with the set of matrices $Q \in \PSD_N^G$ such that exactly one matrix $Q_i$ has rank one, and the other $Q_j$, $j \neq i$ have rank zero, considered up to scaling by $\mathbb{R}_{\geq 0}$.
\end{theorem}
\begin{proof}
Let $Q \in \PSD_N^G$ such that one $Q_i$ has rank one and the others
are zero matrices. The existence of such $Q$ follows from Corollary \ref{cor:nice-diagonal}. 
Without loss of generality we can assume that the $(1,1)$
entry of $Q_i$ is nonzero. We denote this entry by $a$.
Since the columns of $Q_i$ are multiples of the first column
and the rows are multiples of the first row we get
$$Q_i = 
\begin{pmatrix}
a & s_2a & \cdots & s_{m_i}a\\
s_2a & s_2^2a & \cdots & s_2s_{m_i}a\\
\vdots & \vdots & \ddots & \vdots\\
s_{m_i}a & s_{m_i}s_2a & \cdots & s_{m_i}^2a\\
\end{pmatrix}.$$
A basis for $\ker(Q_i)$ is
$$\begin{pmatrix}
-s_2 \\ 1 \\ 0 \\ \vdots \\ 0
\end{pmatrix},
\begin{pmatrix}
-s_3 \\ 0 \\ 1 \\ \vdots \\ 0
\end{pmatrix}, \cdots,
\begin{pmatrix}
-s_{m_i} \\ 0 \\ 0 \\ \vdots \\ 1
\end{pmatrix}.$$
The only symmetric matrices whose kernel contains $\ker(Q_i)$ are scalar multiples of
$$\begin{pmatrix}
1 & s_2 & \cdots & s_{m_i}\\
s_2 & s_2^2 & \cdots & s_2s_{m_i}\\
\vdots & \vdots & \ddots & \vdots\\
s_{m_i} & s_{m_i}s_2 & \cdots & s_{m_i}^2\\
\end{pmatrix}.$$
This also shows that the only symmetric matrices whose kernel
contains $\ker(Q)$ have the same block structure as $Q$ where 
$\tilde{Q}_j = 0$ when $j \neq i$, and in $\tilde{Q_i}$ each 
block is a (possibly different) multiple of $Q_i$. 
But then by Theorem \ref{UnqFaceThm} $S(Q) \cap \PSD_N^G = F_{\PSD_N^G}(Q)$, and this consists of positive multiples of $Q$.  
Therefore the ray generated by $Q$ is an extremal ray.
Any other type of matrix in $\PSD_N^G$ is easily seen to be 
a conical combination of the above matrices. This proves the theorem.
\end{proof}

\begin{cor}
The ranks of extremal rays of $\PSD_N^G$ are precisely $\{n_1, \ldots, n_s\}$, $n_i = \dim V_i$. In particular, the minimum rank attained by extremal matrices
is $\min(n_1, \ldots, n_s)$, and if no one-dimensional representation
of $G$ appears in $V$ with positive multiplicity, this minimum rank is
bigger than one.
\end{cor}

Note that this differs from $\PSD_N$, whose extremal rays are defined by rank one matrices. We continue with Example \ref{ex:icosahedral}.
\begin{continue}{ex:icosahedral}
Consider again the group $G=I_h$ of $120$ symmetries of the icosahedron. The space of degree 3 polynomials has dimension $10$, and can help us write the degree $6$ icosahedral invariants as sums of squares. Using the Mulliken symbols for irreducible representations of $I_h$ typical in chemistry \cite[last page]{Cotton1990ChemicalApplicationsOfGroupTheory}, we have that
\begin{equation*}
    \begin{array}{ccc}
        \text{vector spaces} & \mathbb{C}[x_1,x_2,x_3]_3 & = T_{1u} \oplus T_{2u} \oplus G_u\\
        \text{dimensions} & 10 & = 3 + 3 + 4.
    \end{array}
\end{equation*}
Since the minimum dimension of an irreducible in this decomposition is $3$, we can already conclude that the extremal rays of $\PSD_{10}^{I_h}$ will not be given by matrices of rank 1. The extremal rays correspond to matrices of rank at least $3$.

Similarly, since the degree $5$ polynomials decompose as
\begin{equation*}
    \begin{array}{ccc}
        \text{vector spaces} & \mathbb{C}[x_1,x_2,x_3]_5 & = 2T_{1u} \oplus 2T_{2u} \oplus G_u \oplus H_u \\
        \text{dimensions} & 21 & = 2(3) + 2(3) + 4 + 5,
    \end{array}
\end{equation*}
we know that the extremal rays are defined by matrices of rank exactly $3,4$, and $5$ in $\PSD_{21}^{I_h}$.
\end{continue}

\subsection{Trivial Block}

Here we turn to $G=S_n$ acting on $V = \C[x_1, \ldots, x_n]_d$ by permuting the indices of the indeterminates. For all $n$ and $d$, the 
trivial representation appears in $V$ with multiplicity equal to 
$p=p(n,d)$ where $p(n,d)$ is the number of partitions of $d$ with at most $n$ parts via Theorem \ref{theorem:multiplicity-as-integer-solutions}.
Therefore, in a symmetry adapted basis, there is one $p \times p$ 
diagonal block corresponding to the trivial representation, called the trivial block.

We now use Algorithm \ref{alg:symmetry-adapted-basis} to build the trivial block for any $n$ and $d$ 
in the case of $S_n$. 
Note that we may always use degree $d$ monomials in $n$ variables as a basis for $V$ when $G=S_n$.
To start, we order our monomial basis so that orbits of $G = S_n$ acting on the finite set of monomials  are grouped together. For example, for degree 3 monomials in 3 variables, we could order our basis as 
$$\{x_1^3, x_2^3, x_3^3, x_1^2x_2, x_1^2x_3, x_1x_2^2, x_1x_3^2, x_2^2x_3, x_2x_3^2, x_1x_2x_3 \}$$
which has three orbits $Gv$ for $v \in \{x_1^3, x_1^2 x_2, x_1 x_2 x_3 \}$. 
Note that in general the orbits can be labeled by partitions of $d$ with $ \leq n$ parts.
Under this ordering, a general symmetric matrix will be described by the blocks indexed by the orbits of our monomials
$$Q =
\begin{blockarray}{ccccc}
& O(x^{\lambda^{(1)}}) & O(x^{\lambda^{(2)}}) & \cdots & O(x^{\lambda^{(p)}})   \\
\begin{block}{c(cccc)}
O(x^{\lambda^{(1)}})  &  & & &  \\ 
 O(x^{\lambda^{(2)}})  &  & & &  \\ 
 \vdots & & & & \\
 O(x^{\lambda^{(p)}})  &  & & &  \\ 
\end{block}
\end{blockarray}.$$
\begin{prop} \label{prop:trivial-block}
Let $Q \in \PSD_N^{S_n}$ be an $N \times N$ symmetric matrix represented
in the monomial basis ordered with respect to the orbits $O(x^{\lambda^{(1)}}), O(x^{\lambda^{(2)}}), \ldots, O(x^{\lambda^{(p)}})$. Let	$\Lambda_{i,j}$ be the  submatrix  of $Q$ indexed by $O(x^{\lambda^{(i)}})$ and $O(x^{\lambda^{(j)}})$ on the rows and columns, respectively. Let $s_i = \sqrt{|O(x^{\lambda^{(i)}})|}$. Then there exists an orthogonal change of basis
matrix $T$ such that the trivial block of $T^TQT$ is
$$Q^{\ytableausetup
{mathmode, boxsize=.4em}
\begin{ytableau}
*(white) & *(white) &  \none[\scriptsize\cdot] &  \none[\scriptsize\cdot] &  \none[\scriptsize\cdot] & *(white) \\
\end{ytableau}} =  
\begin{pmatrix}
\frac{s_1^2}{s_1^2}\text{colsum}(\Lambda_{1,1}) & \frac{s_2^2}{s_1s_2}\text{colsum}(\Lambda_{1,2}) & \cdots & \frac{s_p^2}{s_1s_p}\text{colsum}(\Lambda_{1,p})\\
\frac{s_1^2}{s_1s_2}\text{rowsum}(\Lambda_{1,2}) & \frac{s_2^2}{s_2^2}\text{colsum}(\Lambda_{2,2}) & \cdots & \frac{s_p^2}{s_2s_p}\text{colsum}(\Lambda_{2,p})\\
\vdots & \vdots & \ddots & \vdots \\
\frac{s_1^2}{s_1s_p}\text{rowsum}(\Lambda_{1,p}) & \frac{s_2^2}{s_2s_p}\text{rowsum}(\Lambda_{2,p}) & \cdots & \frac{s_p^2}{s_p^2}\text{colsum}(\Lambda_{p,p})\\
\end{pmatrix}$$
where $\text{colsum}(\Lambda_{i,j})$ is the sum of the entries of any column of $\Lambda_{i,j}$ and $\text{rowsum}(\Lambda_{i,j})$ is the sum of the entries of any row of $\Lambda_{i,j}$.  
\end{prop}
\begin{proof}
We follow Algorithm \ref{alg:symmetry-adapted-basis}. Since $n_1 = 1$,
only the very first step needs to be executed. Moreover, $d^1(g) = [1]$
for all $g \in S_n$, and $D(g)$ are block diagonal in the basis 
given by the orbits for each $g$. Hence $\pi^1$ is block diagonal with $p$ blocks of size  $s_i^2 \times s_i^2, i=1, \ldots p$, along the diagonal. 
It is not hard to see that block $i$  is a multiple of the 
$s_i^2 \times s_i^2$ matrix with every entry equal to one. 
Therefore the first $p$ columns of $T$ are
$$T = 
\begin{blockarray}{ccccccc}
\begin{block}{c(cccccc)}
&  &  &  &   \\
 O(x^{\lambda^{(1)}})  & \overline{1/s_1} & \overline{0} & \cdots & \overline{0} & &\\ 
 O(x^{\lambda^{(2)}})  & \overline{0} & \overline{1/s_2} &  \cdots  & \overline{0} & & \\ 
 & \vdots & \vdots & \ddots & \vdots & \cdots & \cdots\\
 O(x^{\lambda^{(p)}})  & \overline{0} & \overline{0} &  \cdots  & \overline{1/s_p} & & \\
\end{block}
\end{blockarray}$$
with the bar indicating a column vector. Now in $T^TQT$
the trivial block has the stated form. 
\end{proof}

\section{Binary and Quadratic Symmetric Polynomials}
\label{section:binary-quadratic}

In this section we first  fix the number of variables $n = 2$ and consider the structure of the symmetry adapted $\PSD_N^G$ cone. In this case, the matrices have size $N=d+1$. We choose the monomial basis  $\{x^d, x^{d-1}y, \ldots, xy^{d-1}, y^d\}$,
and the symmetric matrices will be $Q = (q_{ij})$.
Moreover, we restrict to matrices $Q$ such that $Q D(\sigma) = D(\sigma) Q$ where $\sigma = (1 \, 2)$.

\begin{cor} \label{cor:binary-dim}
When $n=2$ the dimension of the symmetry adapted PSD cone is
$$\dim \PSD_{d+1}^{S_2} = \begin{cases}
\frac{(d+1)(d+3)}{4} & d \text{ is odd} \\
\frac{(d+2)^2}{4} & d \text{ is even} 
\end{cases}$$
\end{cor}
\begin{proof}
The hook lengths are $h_1=2$ and $h_2=1$ for both partitions of $n=2$
corresponding to the trivial and alternating representations. 
Furthermore, $n(\ytableausetup{boxsize=.4em}\ydiagram{2}) = 0$ and $n(\ytableausetup{boxsize=.4em}\ydiagram{1,1}) = 1$.
Thus we can fill out the following table,

\begin{center}
  \begin{tabular}{c | c | c | c  }
  $d$ & Partition & $h^Ty = d - n(\lambda)$ & $m_\lambda$ \\ \hline
		odd & \ytableausetup{boxsize=.4em}\ydiagram{2} & $y_1 + 2y_2 = d$ & $\frac{d+1}{2}$   \\
		  & \ytableausetup{boxsize=.4em}\ydiagram{1,1} & $y_1 + 2y_2 = d-1$ & $\frac{d+1}{2} $  \\ \hline
		even & \ytableausetup{boxsize=.4em}\ydiagram{2} & $y_1 + 2y_2 = d$ & $\frac{d}{2}+1$   \\
		  & \ytableausetup{boxsize=.4em}\ydiagram{1,1} & $y_1 + 2y_2 = d-1$ & $\frac{d}{2}$   \\ \hline
  \end{tabular}
\end{center}
By Corollary \ref{dimcor} we need to compute
$$\dim \PSD_{d+1}^{S_2} = \binom{m_{\ytableausetup{boxsize=.3em}\ydiagram{2}} + 1}{2} + \binom{m_{\ytableausetup{boxsize=.3em}\ydiagram{1,1}} + 1}{2},$$
and this gives the result.
\end{proof}
\begin{prop} \label{prop:symmetric-binary}
There exists a change of basis matrix so that every $Q \in \PSD_N^{S_2}$
with $N = d+1$ is of the form 
$$\frac{1}{2}\begin{pmatrix}
    Q^{\ytableausetup{boxsize=.4em}\ydiagram{2}} & \\
    & Q^{\ytableausetup{boxsize=.4em}\ydiagram{1,1}}
\end{pmatrix}$$
where if $d$ is odd
$$Q^{\ytableausetup{boxsize=.4em}\ydiagram{2}} = \begin{pmatrix}
q_{11} + q_{1N} & q_{12}+q_{1(N-1)} &\cdots & q_{1\frac{N}{2}} + q_{1(\frac{N}{2}+1)}  \\   
q_{12}+q_{1(N-1)} & q_{22}+q_{2(N-1)} &\cdots & q_{2\frac{N}{2}} + q_{2(\frac{N}{2}+1)}  \\   
\vdots & \vdots & \ddots & \vdots\\
q_{1\frac{N}{2}}+q_{1(\frac{N}{2}+1)}&q_{2\frac{N}{2}}+q_{2(\frac{N}{2}+1)}&\cdots &q_{\frac{N}{2}\frac{N}{2}}+q_{\frac{N}{2}(\frac{N}{2}+1)}
\end{pmatrix}$$
and 
$$Q^{\ytableausetup{boxsize=.4em}\ydiagram{1,1}} = \begin{pmatrix}
q_{11} - q_{1N} & q_{12}-q_{1(N-1)} &\cdots & q_{1\frac{N}{2}} - q_{1(\frac{N}{2}+1)}  \\   
q_{12}-q_{1(N-1)} & q_{22}-q_{2(N-1)} &\cdots & q_{2\frac{N}{2}} - q_{2(\frac{N}{2}+1)}  \\   
\vdots & \vdots & \ddots & \vdots\\   
q_{1\frac{N}{2}}-q_{1(\frac{N}{2}+1)}&q_{2\frac{N}{2}}-q_{2(\frac{N}{2}+1)}&\cdots &q_{\frac{N}{2}\frac{N}{2}}-q_{\frac{N}{2}(\frac{N}{2}+1)}
\end{pmatrix},$$
while if $d$ is even there are an extra row and column in the trivial block
$$Q^{\ytableausetup{boxsize=.4em}\ydiagram{2}} = \begin{pmatrix}
q_{11} + q_{1N} & q_{12}+q_{1(N-1)} &\cdots & q_{1\frac{N-1}{2}}+q_{1\frac{N+3}{2}} & \sqrt{2}q_{1\frac{N+1}{2}} \\   
q_{12}+q_{1(N-1)} & q_{22}+q_{2(N-1)} &\cdots & q_{2\frac{N-1}{2}}+q_{2\frac{N+3}{2}} &  \sqrt{2}q_{2\frac{N+1}{2}} \\   
 \vdots & \vdots &\ddots  & \vdots  & \vdots \\   
q_{1\frac{N-1}{2}}+q_{1\frac{N+3}{2}}&q_{2\frac{N-1}{2}}+q_{2\frac{N+3}{2}}&\cdots &q_{\frac{N-1}{2}\frac{N-1}{2}}+q_{\frac{N-1}{2}\frac{N+3}{2}} &\sqrt{2}q_{\frac{N-1}{2}\frac{N+1}{2}} \\
\sqrt{2}q_{1\frac{N+1}{2}}& \sqrt{2}q_{2\frac{N+1}{2}}&\cdots&\sqrt{2}q_{\frac{N-1}{2}\frac{N+1}{2}} & q_{\frac{N+1}{2}\frac{N+1}{2}}
\end{pmatrix}$$
and 
$$Q^{\ytableausetup{boxsize=.4em}\ydiagram{1,1}} = \begin{pmatrix}
q_{11} - q_{1N} & q_{12}-q_{1(N-1)} &\cdots & q_{1\frac{N-1}{2}}-q_{1\frac{N+3}{2}}  \\   
q_{12}-q_{1(N-1)} & q_{22}-q_{2(N-1)} &\cdots & q_{2\frac{N-1}{2}}-q_{2\frac{N+3}{2}}  \\   
\vdots & \vdots & \ddots & \vdots\\      
q_{1\frac{N-1}{2}}-q_{1\frac{N+3}{2}}&q_{2\frac{N-1}{2}}-q_{2\frac{N+3}{2}}&\cdots &q_{\frac{N-1}{2}\frac{N-1}{2}}-q_{\frac{N-1}{2}\frac{N+3}{2}}
\end{pmatrix}.$$
\end{prop}
\begin{proof}
Again we follow Algorithm \ref{alg:symmetry-adapted-basis} where
$d^1(g) = [1]$ and  $d^2(g) = [\mathrm{sign}(g)]$ for $g \in S_2$, $D(\mathrm{id}) = I_{d+1}$
and $$D(12) = \begin{pmatrix} 
0 &  \cdots & 0 & 1 \\
0 & \cdots & 1 & 0 \\
\vdots & \iddots & \vdots &  \vdots\\    
1 & \cdots & 0 & 0   
\end{pmatrix}.$$

Then $\pi^{\ytableausetup{boxsize=.4em}\ydiagram{2}} = I_{d+1} + D(12)$
and $\pi^{\ytableausetup{boxsize=.4em}\ydiagram{1,1}} = I_{d+1} - D(12)$.
The change of basis matrix $T$ looks a little different depending on
the parity of $d$:

\begin{center}
  \begin{tabular}{ c | c }
    $d$ odd & $d$ even \\ \hline
$\frac{\sqrt{2}}{2}\begin{pmatrix}
1 & 0 & &   & 1   & 0 &  &  \\   
0 & 1 & & & 0 &  1&  &  \\   
 &  &  \vdots& & & & \vdots &   \\   
& & & 1 & 0 & & & 1\\
& & & 1 & 0 & & & -1\\
 &  &  \vdots& & & & \vdots &   \\   
0 & 1 & & &0&  -1 &   \\   
1 & 0 & &  &-1  & 0&   &      
\end{pmatrix}$ &
$\frac{\sqrt{2}}{2}\begin{pmatrix}
1 & 0 & &   & 1 & 0    &  \\   
0 & 1 &  & &  0 &1&  \\   
 &  &  \vdots&  & &&  \vdots &   1 \\   
& & & \sqrt{2}&  & && 0 \\
 &  &  \vdots&  &&&  \vdots &   -1\\   
0 & 1 & & & 0 & -1 && \\   
1 & 0 & &  & -1 &  0 &     \end{pmatrix}$
  \end{tabular}
\end{center}
By computing $T^TQT$  we get $Q^{\ytableausetup{boxsize=.4em}\ydiagram{2}}$ and $Q^{\ytableausetup{boxsize=.4em}\ydiagram{1,1}}$
in both cases.
\end{proof}

We briefly consider an example to which we will return in Section \ref{sec:inequalities}. 
\begin{ex}
Consider the symmetric polynomial inequality $P_{4}(x,y) \geq P_{1111}(x,y)$ where $P_{4}(x,y) = \frac{1}{2} (x^4+y^4)$
and $P_{1111}(x,y)= \frac{1}{16}(x+y)^4$.
It is proven in \cite{CGS2011} that this inequality holds over the nonnegative orthant. We can certify this inequality via sums of squares.
First, define the polynomial, 
$$f(x,y) = (P_{4} - P_{1111})(x^2,y^2) = \frac{1}{2}(x^8+y^8)-\frac{1}{16}(x^2+y^2)^4 $$
$$=\frac{7}{16} \, x^{8} - \frac{1}{4} \, x^{6} y^{2} - \frac{3}{8} \, x^{4} y^{4} - \frac{1}{4} \, x^{2} y^{6} + \frac{7}{16} \, y^{8}$$
and note that if $f$ is SOS, then the above inequality holds for $x,y \geq 0$. Next, assume that $f= m(x)^TQm(x)$ where $Q = (q_{ij})$ is a $5\times 5$
symmetric matrix in the monomial basis $m(x)^T = [x^4, x^3y, x^2y^2, xy^3, y^4]$.
We equate the coefficients of $f(x,y)$ and 
\begin{align*}
    m(x)^TQm(x) = q_{11}x^8 + 2q_{12}x^7y &+ (2q_{13}+ q_{22})x^6y^2 + (2q_{14} + 2q_{23})x^5y^3 + (2q_{15} + 2q_{24} + q_{33})x^4y^4 \\
    &+ (2q_{14} + 2q_{23})x^3y^5 + (2q_{13} + q_{22})x^2y^6 + 2q_{12}xy^7 + q_{11}y^8
\end{align*}
to find out $q_{11} = \frac{7}{16}$, $q_{12} = 0$, $q_{22} = -\frac{1}{4}-2q_{13}$, $q_{23} = -q_{14}$, and $q_{33} = -\frac{3}{8} - 2q_{15} - 2q_{24}$. Substituting these into the computed matrices
in Proposition \ref{prop:symmetric-binary} for $d=4$ (even), 
our matrix $Q$ becomes
$$\left(\begin{array}{rrrrr}
q_{15} + \frac{7}{16} & q_{14} & \sqrt{2} q_{13} & 0 & 0 \\
q_{14} & -2 \, q_{13} + q_{24} - \frac{1}{4} & -\sqrt{2} q_{14} & 0 & 0 \\
\sqrt{2} q_{13} & -\sqrt{2} q_{14} & -2 \, q_{15} - 2 \, q_{24} - \frac{3}{8} & 0 & 0 \\
0 & 0 & 0 & -q_{15} + \frac{7}{16} & -q_{14} \\
0 & 0 & 0 & -q_{14} & -2 \, q_{13} - q_{24} - \frac{1}{4}
\end{array}\right).$$
Now we can run an SDP on this to certify that $f$ is SOS.  
One rank two solution is
$$\left(\begin{array}{rrrrr}
\frac{7}{8} & 0 & -\frac{7\sqrt{2}}{8}  & 0 & 0 \\
0 & 0 & 0 & 0 & 0 \\
-\frac{7\sqrt{2}}{8}  & 0 & \frac{7}{4} & 0 & 0 \\
0 & 0 & 0 & 0 & 0 \\
0 & 0 & 0 & 0 & 3
\end{array}\right). $$
This is indeed positive semidefinite and thus $f(x,y) = (P_{4} - P_{1111})(x^2,y^2)$ is SOS.
\end{ex}

In the rest of the section we consider symmetric quadratic polynomials ($d=1$) in any number of variables $n$. Then $N=n$ and $D(g)$ are
the $n \times n$ permutation matrices represented in the 
monomial basis $\{x_1, x_2, \ldots, x_n\}$. It is not hard to see 
that in this basis all $n \times n$ symmetric matrices which 
commute with all permutation matrices are  of the form
$$ Q = \begin{pmatrix}
q_{11} & q_{12} &\cdots & q_{12}  \\   
q_{12} & q_{11} &\cdots & q_{12}  \\   
\vdots & \vdots & \ddots & \vdots \\   
q_{12} & q_{12} &\cdots &q_{11}
\end{pmatrix}.$$
In a symmetry adapted basis the matrices look even simpler.
\begin{prop} \label{prop:symmetric-quadratic}
There is a change of basis matrix such that every $Q \in \PSD_n^{S_n}$ is of the form
$$ \begin{pmatrix}
q_{11} + (n-1)q_{12} & 0 &\cdots & 0  \\   
0 & q_{11} - q_{12} &\cdots & 0  \\   
\vdots & \vdots & \ddots & \vdots \\   
0 & 0 &\cdots & q_{11} - q_{12}
\end{pmatrix}.$$
\end{prop}
\begin{proof}
The representation in question is the permutation representation of $S_n$.
By Theorem \ref{theorem:multiplicity-as-integer-solutions}, the trivial representation and the standard representation both appear with multiplicity one. This tells us that there will be one $1\times 1$ block associated to the trivial representation, and $n-1$ copies of a $1\times1$ block associated to the standard representation. The application of Algorithm $\ref{alg:symmetry-adapted-basis}$ yields
the desired diagonal matrix.
\end{proof}

\begin{cor}
An $n\times n$ symmetric matrix $Q = (q_{ij})$ which commutes with $D((1\,2))$ is in $\PSD_n^{S_n}$
if and only if $q_{12} \leq q_{11}$ and $q_{12} \geq \frac{-1}{n-1}q_{11}$. Hence $\PSD_n^{S_n}$ is a two-dimensional
polyhedral cone defined by these linear inequalities.
\end{cor}

\begin{cor}
Let $f(x) = a\sum_i x_i^2 + b\sum_{i<j} x_ix_j$ be a symmetric quadratic
form. Then $f$ is SOS if and only if $ \frac{-1}{n-1} a \leq b \leq a$.
Moreover, the symmetry adapted Gram spectrahedron $K_f^{S_n}$ is 
either empty or an isolated point in $\PSD_n^{S_n}$.
\end{cor}
\begin{proof}
Observe that $a=q_{11}$ and $b=q_{12}$ by (\ref{equation:sos-representation-by-monomial-gram-monomial}). Clearly, $K_f^{S_n} = \{(a,b)\}$
if and only if $f$ is SOS. 
\end{proof}
\begin{cor}\label{corollary:quadratic-SOS-one-nminus1-or-n}
Symmetric quadratic SOS forms can only be written as 
a sum of one, $n-1$, or $n$ squares. 
\end{cor}
\begin{proof}
Let $f(x) = a\sum_i x_i^2 + b\sum_{i<j} x_ix_j$ be a symmetric quadratic form and consider (\ref{equation:sos-representation-by-monomial-gram-monomial}) with $Q\in \PSD_N$,
$$a\sum_i x_i^2 + b\sum_{i<j} x_ix_j = \begin{pmatrix}
x_1 & \cdots & x_n
\end{pmatrix}
\begin{pmatrix}
q_{11} & q_{12} & \cdots & q_{1n} \\
q_{12} & q_{22} & \cdots & q_{2n} \\
\vdots &\vdots &\ddots & \cdots \\
q_{1n} & q_{2n} & \cdots & q_{nn} \\
\end{pmatrix}
\begin{pmatrix}
x_1 \\ \vdots \\ x_n
\end{pmatrix}.$$
Equating coefficients we see that $a = q_{ii}$ and $b = q_{ij}$ for $i\neq j$, the same structure as any invariant matrix.
Thus  $\PSD_n^{S_n}$ is in fact representative of all SOS decompositions of symmetric quadratic SOS forms.
Now, if the point $(a,b)$ is in the interior of $\PSD_n^{S_n}$, the corresponding matrix has full rank.  
If it is on the extreme ray defined by $q_{12} = q_{11}$, the matrix rank will be one as all the blocks $q_{12} - q_{11}$ will be zero.
Lastly, if it is on the other extreme ray, we get a rank $n-1$ matrix.
\end{proof}

Finally, we consider what happens as the number of variables goes to infinity. In particular, note that the slope of $q_{12} = \frac{-1}{n-1}q_{11}$ goes to zero. This leads to the following result.
\begin{theorem}
As $n$ goes to infinity, the ratio of SOS symmetric quadratic forms in $n$ variables to all symmetric quadratic forms in $n$ variables is $\frac{1}{8}$.
\end{theorem}

\section{Ternary Symmetric Polynomials}
\label{section:ternary-symmetric}

In this section we consider the case where $n = 3$ and $N = \frac{1}{2}(d+2)(d+1)$.
\begin{prop} \label{prop:ternary-multiplicities}
Let $V = \C[x_1,x_2,x_3]_d$ be the representation of $S_3$ induced
by permuting the variables. Then the multiplicities of the 
trivial, standard, and alternating irreducible representations
are as in the following table
\begin{center}
  \begin{tabular}{ c | c | c}
   Partition & $h^Ty = d - n(\lambda)$ & $m_\lambda$  \\ \hline
		\ytableausetup{boxsize=.4em}\ydiagram{3} & $y_1 + 2y_2 + 3y_3 = d$ & $Q(d)$ \\
		\ytableausetup{boxsize=.4em}\ydiagram{2,1} & $y_1 + y_2 + 3y_3  = d-1$ & $P(d-1)$   \\ 
	    \ytableausetup{boxsize=.4em}\ydiagram{1,1,1} & $y_1 + 2y_2 + 3y_3 = d-3$  & $Q(d-3)$ \\
  \end{tabular}
\end{center}
where $Q(d)$ and $P(d)$ are quasi-polynomials as below:
$$Q(d) = \begin{cases}
\frac{1}{12}d^2 + \frac{1}{2}d + 1 & d \equiv 0 \mod 6\\
\frac{1}{12}d^2 + \frac{1}{2}d + \frac{5}{12} & d \equiv 1 \mod 6\\ 
\frac{1}{12}d^2 + \frac{1}{2}d + \frac{2}{3} & d \equiv 2 \mod 6\\
\frac{1}{12}d^2 + \frac{1}{2}d + \frac{3}{4} & d \equiv 3 \mod 6\\
\frac{1}{12}d^2 + \frac{1}{2}d + \frac{2}{3} & d \equiv 4 \mod 6\\
\frac{1}{12}d^2 + \frac{1}{2}d + \frac{5}{12} & d \equiv 5 \mod 6\\
\end{cases}   \quad \quad
P(d) = \begin{cases}
\frac{1}{6}d^2 + \frac{5}{6}d + 1 & d \equiv 0 \mod 3\\ 
\frac{1}{6}d^2 + \frac{5}{6}d + 1 & d \equiv 1 \mod 3\\
\frac{1}{6}d^2 + \frac{5}{6}d  + \frac{2}{3} & d \equiv 2 \mod 3\\
\end{cases}$$
\end{prop}
\begin{proof}
The multiplicities are computed using Theorem \ref{theorem:multiplicity-as-integer-solutions}. In all three cases,
they are given by the Ehrhart quasi-polynomial \cite{BR2015}
of a rational $2$-simplex. For instance, for the trivial 
representation we wish to count the number of nonnegative integer solutions to the equation $y_1 + 2y_2 + 3y_3 = d$.
This is the number of lattice points in the polytope defined by the hyperplane $y_1 + 2y_2 + 3y_3 = d$ 
and $y_1, y_2, y_3 \geq 0$. The vertices  of this polytope are 
$(d,0,0), (0,d/2,0), (0,0,d/3)$, and it is  
the $d$th dilation of the polytope for $d = 1$. 
The lattice point count is given by the quasi-polynomial $Q(d)$ as
in the statement. Similarly, for the multiplicity of the 
standard representation, the Ehrhart quasi-polynomial $P(d)$ of 
a different two-simplex is needed. 
\end{proof}

\subsection{Symmetric Ternary Quartics}
Now we consider symmetric polynomials in three variables of degree four ($n=3, d=2$).
The study of general ternary quartics has a long history. 
It is known that a smooth ternary quartic can always be written
as $f = q_1^2 + q_2^2 + q_3^2$ where $q_i \in \C[x_1,x_2,x_3]_2$, 
and there are exactly $63$ nonequivalent ways of doing that 
\cite[Ch.1, \S 14]{Coble}. There are always $28$
bitangents to the smooth projective plane curve defined by $f$, and certain sixtuples of pairs of these bitangents, known as 
\textit{Steiner complexes}, correspond to these $63$ different representations; see \cite[Section 5]{PSV2011}. Moreover,  
for real smooth ternary quartics there are exactly $8$ SOS representations
with three squares \cite{PRSS2004}. This means that the usual 
Gram spectrahedron $K_f$ has exactly $8$ vertices corresponding to
matrices of rank $3$. 

In this section, we want to study the symmetry adapted Gram spectrahedron $K_f^{S_3}$. The main objects of focus are the symmetric 
matrices $Q = (q_{ij}) \in \Sym^6$ such that $f(x) = m(x)^TQm(x)$.

\begin{prop} \label{prop:symmetry-adapted-ternary-quartic}
The symmetry adapted $\PSD_6^{S_3}$ is a six-dimensional cone
consisting of positive semidefinite matrices of the form
$$\begin{pmatrix}
q_{11} + 2q_{12} & 2q_{14} + q_{16} & 0 & 0 & 0 & 0 \\
2q_{14} + q_{16} & q_{44} + 2q_{45} & 0 & 0 & 0 & 0 \\
0 & 0 & q_{11} - q_{12} & q_{14} - q_{16} & 0 & 0 \\
0 & 0 & q_{14} - q_{16} &  q_{44} - q_{45} & 0 & 0 \\
0 & 0 & 0 & 0 & q_{11} - q_{12} & q_{14} - q_{16} \\
0 & 0 & 0 & 0 & q_{14} - q_{16} & q_{44} - q_{45}
\end{pmatrix}.$$
\end{prop}
\begin{proof}
Proposition \ref{prop:ternary-multiplicities} tells us that
the multiplicities of the trivial and standard representations are each two, and that of the alternating representation is zero. By Corollary 
\ref{dimcor} the dimension of $\PSD_6^{S_3}$ is six. Using Algorithm
\ref{alg:symmetry-adapted-basis}, one can compute a $6 \times 6$ change
of basis matrix such that the elements in $\PSD_6^{S_3}$ have
the stated form. 
\end{proof}

The next theorem is our main theorem in this section. 
\begin{theorem}
Let $f\in \R[x, y, z]$ be a smooth symmetric quartic. Then there are precisely $3$ (possibly complex) symmetric matrices $Q$ of rank $3$ such that $f= m(x)^TQm(x)$ and $D(\sigma) Q = QD(\sigma)$ for all $\sigma \in S_3$. 
Moreover, if $f$ is SOS,
there are exactly $2$ such PSD matrices of rank $3$. These correspond
to the two vertices of the two-dimensional symmetry adapted Gram spectrahedron $K_f^{S_3}$. Furthermore, the boundary of $K_f^{S_3}$ is defined by two curves, a parabola and a hyperbola.  
Other than the two vertices, the points along the hyperbola give rank $4$ matrices while those along the parabola are rank $5$ matrices.
\label{thm:symmetric-ternary-quartic-main}
\end{theorem}
\begin{proof}
Let 
$$f(x_1,x_2,x_3) = a\sum_i x_i^4 + b\sum_{i\neq j} x_i^3x_j + c\sum_{i < j}x_i^2x_j^2 + d\sum_{i \neq j \neq k, j < k}x_i^2x_jx_k$$
where $a,b,c,d$ are fixed coefficients. Writing $f = m(x)^TQm(x)$
and equating coefficients we get that 
$a = q_{11}$, $b = 2q_{14}$, $c = 2q_{12} + q_{44}$, and $d = 2q_{16} + 2q_{45}$. If we plug these into the block-diagonalized 
matrix in Proposition \ref{prop:symmetry-adapted-ternary-quartic}
we see that the symmetry adapted Gram spectrahedron $K_f^{S_3}$
consists of positive semidefinite matrices of the form 
$$\begin{pmatrix}
a + 2q_{12} & b + q_{16} & 0 & 0 & 0 & 0 \\
b + q_{16} & c + d - 2q_{12} - 2q_{16} & 0 & 0 & 0 & 0 \\
0 & 0 & a - q_{12} & \frac{b}{2} - q_{16} & 0 & 0 \\
0 & 0 & \frac{b}{2} - q_{16} &  c - \frac{d}{2} - 2q_{12}  + q_{16} & 0 & 0 \\
0 & 0 & 0 & 0 & a - q_{12} & \frac{b}{2} - q_{16} \\
0 & 0 & 0 & 0 & \frac{b}{2} - q_{16} & c - \frac{d}{2} - 2q_{12}  + q_{16}
\end{pmatrix}.$$
Hence, $K_f^{S_3}$ is the intersection of two spectrahedra:
\begin{equation}\label{equ:ternary_quartic_trivial_block}
    K_1 \, = \, \{(q_{12},q_{16}) : \begin{pmatrix}
    a + 2q_{12} & b + q_{16} \\
    b + q_{16} & c + d - 2q_{12} - 2q_{16}
    \end{pmatrix} \succeq 0\}
\end{equation}
\begin{equation}\label{equ:ternary_quartic_standard_block}
    K_2 \, = \, \{(q_{12},q_{16}) :  \begin{pmatrix}
    a - q_{12} & \frac{b}{2} - q_{16} \\
    \frac{b}{2} - q_{16} & c - \frac{d}{2} - 2q_{12}  + q_{16}
    \end{pmatrix} \succeq 0\}
\end{equation}
To prove the first statement in our theorem we ignore the condition
that these matrices need to be positive semidefinite. The 
above $6 \times 6$ matrix has rank three if and only if the two 
$2 \times 2$ matrices have rank one. 
Thus their determinants must be zero.
This gives us two quadratics in the variables $q_{12}$ and $q_{16}$ which we homogenize using a new variable $q$:
\begin{align*}
    p_1 &= -4q_{12}^2 -4q_{12}q_{16} - q_{16}^2 + q((-2a + 2c + 2d)q_{12} + (- 2a - 2b)q_{16}) + q^2(ac + ad-b^2)\\
    p_2 &= 2q_{12}^2 -q_{12}q_{16} - q_{16}^2 + q((-2a- c + d/2)q_{12} + (a + b)q_{16}) + q^2(ac - ad/2 -b^2/4).
\end{align*}
By Bezout's theorem, the projective plane curves defined
by $p_1$ and $p_2$ intersect at exactly $4$ complex points. 
Setting $q=0$, we consider the solutions to the equations
\begin{align*}
    0 &= -4q_{12}^2 -4q_{12}q_{16} - q_{16}^2 = -(2q_{12}+q_{16})^2\\
    0 &= 2q_{12}^2 -q_{12}q_{16} - q_{16}^2 = (2q_{12} + q_{16})(q_{12} - q_{16}).
\end{align*}
We see that there is only one solution, giving us $[q_{12} : q_{16}:  q] = [1:-2:0]$ as the intersection point at the line at infinity. The remaining three points are obtained by setting  $q=1$ which gives us back the determinants of the two submatrices. This proves the first statement.

Next we consider the spectrahedra $K_1$ and $K_2$.
For fixed $a,b,c,d$, $K_1$ is defined by the inequalities
\begin{align*}
    (a + 2q_{12})(c + d - 2q_{12} - 2q_{16}) - (b + q_{16})^2 &\geq 0\\
    a + 2q_{12} &\geq 0\\
    c + d - 2q_{12} - 2q_{16} &\geq 0. 
\end{align*}
The first quadratic can be rewritten as
$$\begin{pmatrix}q_{12} & q_{16} & 1 \end{pmatrix}
\begin{pmatrix}
-4 & -2 & -a+c+d \\
-2 & -1 & -a-b  \\
-a+c+d & -a-b & ac + ad-b^2
\end{pmatrix}
\begin{pmatrix}q_{12} \\ q_{16} \\ 1 \end{pmatrix} \geq 0.$$
Since the determinant of the upper left $2\times2$ matrix is zero, the curve defined by this quadric is a parabola \cite[Table 5.3]{G1998}.  
Moreover, the lines $a + 2q_{12} = 0$ and $c + d - 2q_{12} - 2q_{16} = 0$ are tangent to the curve at the points 
$(-\frac{a}{2},-b)$ and $(b + \frac{c}{2} + \frac{d}{2},-b)$ respectively.
As we vary $a,b,c,d$, the region defined by the first inequality moves between only two of the four connected components in the complement of the two lines as illustrated below:
\begin{figure}[H]
	\centering
		\includegraphics[width=0.8\textwidth]{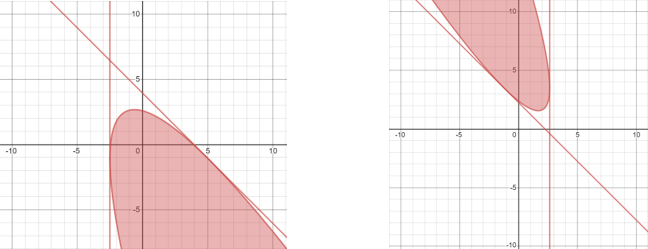}
	\label{fig:Parabolas}
\end{figure}
Moreover, by the last two inequalities, $K_1$ is nonempty when the parabola is in the bottom region, as in the left most figure.  
It is worth noting that this is the generic case and that there is one more possibility.  
If the determinant of the above matrix is zero, i.e., $(a + 2b + c + d)^2 = 0$,
then the quadric defines a double line \cite[Table 5.3]{G1998}, $(a - c - d + 4q_{12} + 2q_{16})^2 = 0$. 
This double line intersects the lines $a + 2q_{12} = 0$ and $c + d - 2q_{12} - 2q_{16} = 0$ at the same point.  Thus $K_1$ is a ray, starting from this intersection point and going out to $(\infty, -\infty)$:
\begin{figure}[H]
	\centering
		\includegraphics{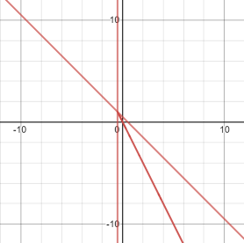}
	\label{fig:ParabolasLine}
\end{figure}
We can do a similar analysis of $K_2$ which is defined by the inequalities
\begin{align*}
    (a - q_{12})(c - \frac{d}{2} - 2q_{12}  + q_{16}) - (\frac{b}{2} - q_{16})^2 &\geq 0\\
    a - q_{12} &\geq 0\\ 
    c - \frac{d}{2} - 2q_{12}  + q_{16} &\geq 0
\end{align*}
We rewrite the first quadratic as
$$\begin{pmatrix}q_{12} & q_{16} & 1 \end{pmatrix}
\begin{pmatrix}
2 & -\frac{1}{2} & -a -\frac{c}{2} + \frac{d}{4} \\
-\frac{1}{2} & -1 & \frac{a}{2}+\frac{b}{2}  \\
-a -\frac{c}{2} + \frac{d}{4} & \frac{a}{2}+\frac{b}{2} & -\frac{b^2}{4} + ac - \frac{ad}{2}
\end{pmatrix}
\begin{pmatrix}q_{12} \\ q_{16} \\ 1 \end{pmatrix}.$$
This is a hyperbola (or a pair of crossing lines) because the leading $2 \times 2$ minor is nonzero \cite[Table 5.3]{G1998}.
Again the two additional inequalities define lines that are tangent to the curve and give $K_2$ as the left most component of the hyperbola:
\begin{figure}[H]
	\centering
		\includegraphics{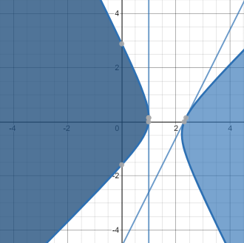}
	\label{fig:Hyp}
\end{figure}
We now see that for a generic symmetric ternary quartic that is SOS, the 
symmetry adapted Gram spectrahedron $K_f^{S_3}$ is the intersection of the parabola and one component of the hyperbola.
\begin{figure}[H]
	\centering
		\includegraphics{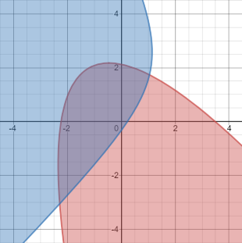}
	\label{fig:QuarticTernaryGramSpec}
\end{figure}
The two points in $K_f^{S_3}$ where these curves intersect are 
the two vertices corresponding to rank three matrices.   
If we move along the boundary defined by the parabola, we get rank $5$ matrices, because on these points the matrix block corresponding to the parabola has rank $1$ while the two blocks corresponding to the hyperbola are each rank $2$.  
A similar argument shows that matrices along the hyperbola have rank $4$.
\end{proof}

\begin{remark}
Theorem \ref{thm:symmetric-ternary-quartic-main} illustrates one of three cases, namely, the case where $f$ is SOS when the two quadrics  defined by the determinants of the matrices in $K_1$ and $K_2$ intersect at three real points, two of which give PSD matrices.  
If $f$ is not SOS, then we have two additional situations.  
The first is that the curves only intersect at one real point and two complex points, and the second case is when the curves have three real intersection points. In the latter, even though there are three real points, none of  them correspond to a PSD matrix.

\begin{figure}[H]
    \centering
    \includegraphics[width=0.3\textwidth]{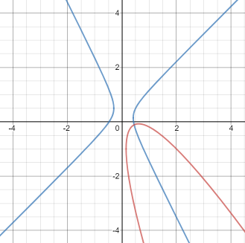} \hspace{1cm} \includegraphics[width=0.30\textwidth]{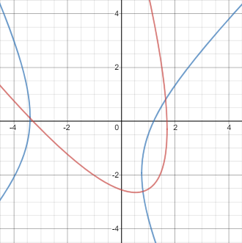}
    \label{fig:ParaHypOneReal}
\end{figure}

\end{remark}

As mentioned above, the Gram spectrahedron of an SOS ternary quartic $f$ has $8$ vertices of rank three.
Let the \emph{Steiner graph} be the graph on these vertices whose edges represent edges of the Gram spectrahedron.
For a generic SOS ternary quartic the Steiner graph is $K_4 \sqcup K_4$, the disjoint union of two complete graphs on $4$ vertices \cite{PSV2011}. Moreover, the matrices along those edges are of rank at most $5$.
It is not known whether the Steiner graph  coincides with all edges of the Gram spectrahedron. However, it is clear from Theorem \ref{thm:symmetric-ternary-quartic-main} that, generically, there are no edges of the symmetry adapted Gram spectrahedron contributing to the edges
of the Steiner graph. 

\begin{cor}
The Steiner graph of the symmetry adapted Gram spectrahedron of a generic symmetric SOS ternary quartic $f$ is the disjoint union of two vertices.
\end{cor}
\begin{proof}
By Theorem \ref{thm:symmetric-ternary-quartic-main}, $K_f^{S_n}$ has two vertices.  Thus, either both vertices are in one complete graph $K_4$ or each graph contains one of the two vertices.  If it were the former, then $K_f^{S_n}$ would also contain the corresponding edge.
This is, however, the interior of the symmetry adapted Gram spectrahedron and all matrices there are rank 6.  Thus no such edge of matrices of rank 5 exists, i.e. the vertices are each in different complete graphs.
\end{proof}

The vertices of the Gram spectrahedron of $f$ or of its symmetry adapted
version when $f$ is $G$-invariant  
are not the end of the 
story. The boundary of these spectrahedra are very interesting and
the work to unearth it is only starting. In the symmetric ternary 
quartics case, the boundary consists of the union of a piece of a
parabola and a piece of a hyperbola. It is an interesting question 
how a typical SOS decomposition would look like 
if we used an SDP solver for $K_f^{S_n}$. It is not difficult to run
simulations. Below are the results of such computations. 
We generated random symmetric ternary quartics and ran SDPs until we found $16$ that were SOS. For each of these $16$ polynomials we randomly generated 1000 objective functions and ran an SDP for each of them. The ranks of the corresponding 1000 optimal SOS matrices are shown in 
Figure \ref{fig:ternary-quartic-simulation}.

\begin{figure}[H] \label{fig:ternary-quartic-simulation}
\begin{center}
    \includegraphics[width=0.8\textwidth]{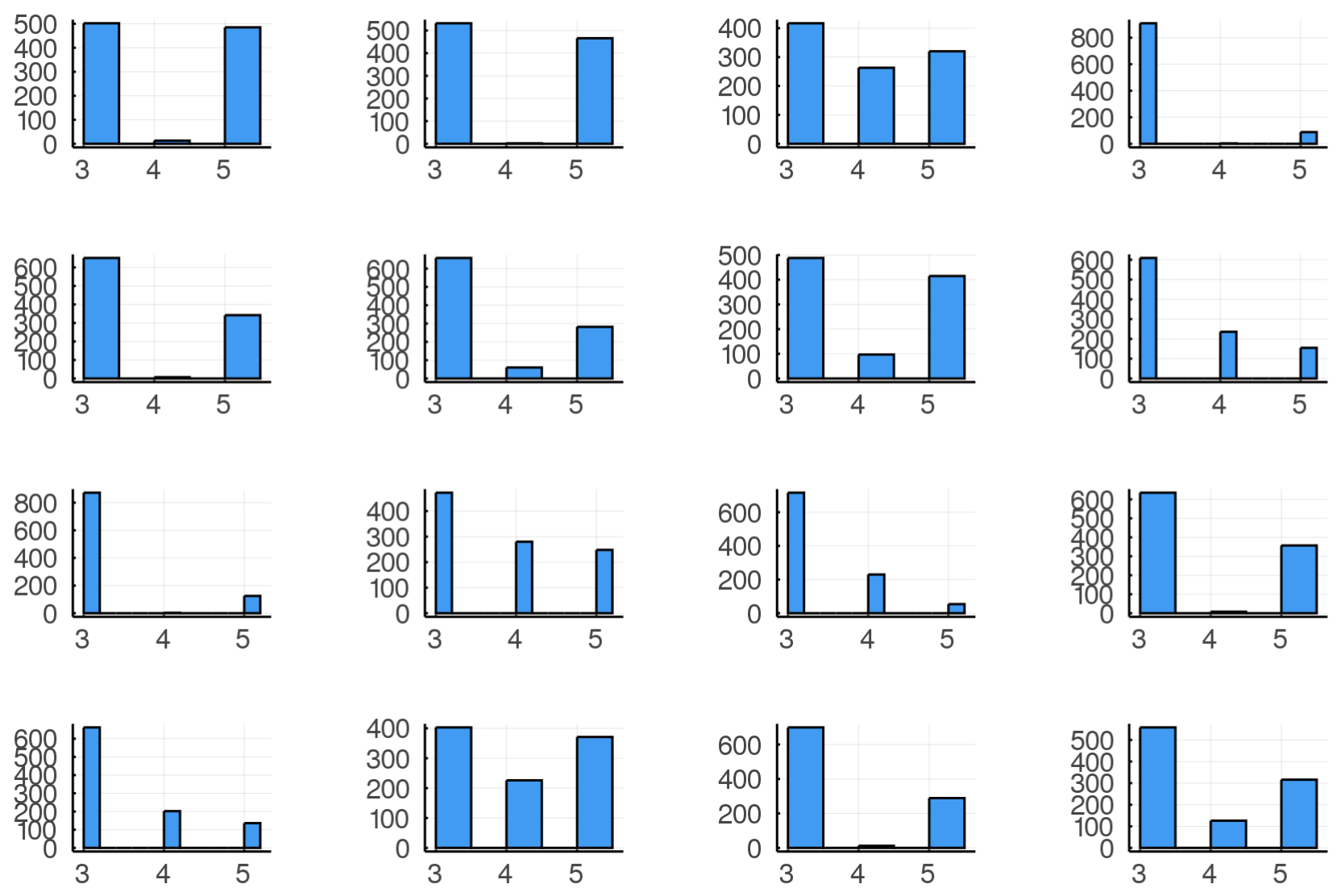}
\end{center}
\caption{Distribution of ranks for SOS decomposition of 16 symmetric ternary
quartics} 
\end{figure}

\begin{remark}
Computational data can provide some insight about the normal fan of the symmetry adapted Gram spectrahedron. In the generic case for a positive ternary quartic, the normal fan will be something like the following:
\begin{figure}[H]
	\centering
		\includegraphics[width=0.25\textwidth]{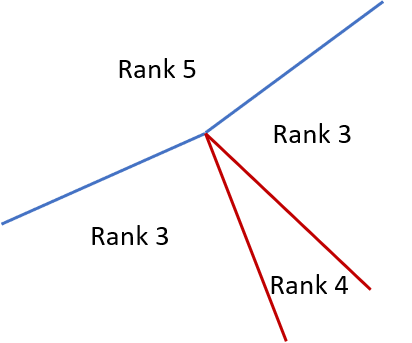}
	\label{fig:NormalFan}
\end{figure} 
\noindent Hence a random cost function is more likely to return a rank three or a rank five solution than a rank 4 solution, as reflected by the data.

\end{remark}

We close this section with a characterization of all symmetric 
ternary quartics that are SOS. First we provide necessary linear 
conditions on the coefficients of such a polynomial. Then we report
on a full characterization in a form which can be used
to certify whether a symmetric ternary quartic is SOS.

\begin{prop}  \label{prop:ternary_quartic_necessary_inequalities}
If a symmetric ternary quartic
$$f(x_1,x_2,x_3) = a\sum_i x_i^4 + b\sum_{i\neq j} x_i^3x_j + c\sum_{i < j}x_i^2x_j^2 + d\sum_{i \neq j \neq k, j < k}x_i^2x_jx_k$$
with real coefficients $a,b,c,d$ is SOS, then
\begin{enumerate}
    \item[a)] $a \geq 0 $,
    \item[b)] $a + c \geq 0$,
    \item[c)] $a + 2b + c + d \geq 0$.
\end{enumerate}
\end{prop}   
\begin{proof}
The first two conditions follow from projecting 
the polyhedron defined by linear inequalities obtained 
from the four diagonals in (\ref{equ:ternary_quartic_trivial_block})
and (\ref{equ:ternary_quartic_standard_block}).
The third condition comes from applying quantifier elimination on the defining inequalities of $K_1$ in (\ref{equ:ternary_quartic_trivial_block}).
\end{proof}

\begin{ex}
As mentioned, the conditions in Proposition \ref{prop:ternary_quartic_necessary_inequalities} are not sufficient.  Let $a=1$, $b=2$, $c=1$, and $d=0$.  Certainly, $a$, $a+c$, and $a + 2b + c + d$ are all nonnegative, but the corresponding polynomial,
$$f(x_1,x_2,x_3) = \sum_i x_i^4 + 2\sum_{i\neq j} x_i^3x_j + \sum_{i < j}x_i^2x_j^2$$
is not SOS.  In particular, $f(1,-2,1) = -9$.
\end{ex}

Additional conditions are not easy to find.  The task is to project 
the spectrahedron $K_1 \cap K_2$ onto the $(a,b,c,d)$-space. 
Given that this $6$-dimensional spectrahedron is a cone, one method is to consider the projection of an affine slice.  We do this for $q_{16} = 1$.
Then for any $(a,b,c,d)$ in this projection, the corresponding polynomial is SOS and so is any positive scaling of that polynomial.  
However, for a complete description, we must also consider the projection when $q_{16} = 0$ and $q_{16} = -1$.
In this way, we can find an exact description (up to positive scaling) of the semialgebraic set defined by the projection of the three affine slices when $q_{16} = 1$, $q_{16} = 0$, and $q_{16} = -1$ using quantifier elimination. The result is the union of $158$ basic semialgebraic sets, each defined with polynomial inequalities and equations up to degree $4$. 
We encourage the interested reader to visit  
\begin{center}
    \href{https://math.berkeley.edu/~ishankar/SOSSymTernQuartic.html}{https://math.berkeley.edu/$\sim$ishankar/SOSSymTernQuartic.html}
\end{center}
for a code that will check if a given point $(a,b,c,d)$ is contained in this set, and thus are the coefficients of an SOS polynomial. There one may also see the full description of the projected slices of the spectrahedron. 

\subsection{Symmetric Ternary Sextics}

Here $V = \R[x_1,x_2,x_3]_3$ and we consider symmetric ternary sextics.

\begin{prop} \label{prop:PSD-ternary-sextics}
The symmetry adapted PSD cone $\PSD_{10}^{S_3}$ consists of  $10 \times 10$ symmetric matrices of the form 
$$\begin{pmatrix}
					Q^{\ytableausetup{boxsize=.4em}\ydiagram{3}}&          & &\\
								 & Q^{\ytableausetup{boxsize=.4em}\ydiagram{2,1}}& & \\
								 &          &Q^{\ytableausetup{boxsize=.4em}\ydiagram{2,1}} & \\
								 &          & & Q^{\ytableausetup{boxsize=.4em}\ydiagram{1,1,1}}\\			
					\end{pmatrix}$$
					where each
					$$Q^{\ytableausetup{boxsize=.4em}\ydiagram{3}} = \begin{pmatrix}
					q_{11}+2q_{12} & \sqrt{2}(q_{14}+q_{16}+q_{18}) & \sqrt{3}q_{110} \\
					\sqrt{2}(q_{14}+q_{16}+q_{18}) & q_{44}+q_{45}+q_{46}+2q_{47}+q_{49} & \sqrt{6}q_{410} \\
					\sqrt{3}q_{110} & \sqrt{6}q_{410} & q_{1010} \\
					\end{pmatrix} $$
					$$Q^{\ytableausetup{boxsize=.4em}\ydiagram{2,1}} = \begin{pmatrix}  
						q_{11}-q_{12} & \frac{\sqrt{2}}{2}(2q_{14}-q_{16}-q_{18}) & \frac{\sqrt{6}}{2}(q_{16}-q_{18})  \\ 
					 \frac{\sqrt{2}}{2}(2q_{14}-q_{16}-q_{18}) & q_{44} + q_{45} - \frac{1}{2}q_{46} - q_{47} - \frac{1}{2}q_{49} &	\frac{\sqrt{3}}{2}(q_{46}-q_{49})  \\ 
					 \frac{\sqrt{6}}{2}(q_{16}-q_{18}) & \frac{\sqrt{3}}{2}(q_{46}-q_{49}) &  q_{44} - q_{45} + \frac{1}{2}q_{46} - q_{47} + \frac{1}{2}q_{49}  \\
						\end{pmatrix} $$
					$$Q^{\ytableausetup{boxsize=.4em}\ydiagram{1,1,1}} = q_{44} - q_{45} - q_{46} + 2q_{47} - q_{49}$$
is positive semidefinite.
\end{prop}
\begin{proof}
The multiplicities of the trivial, standard, and alternating irreducible representations are three, three, and one, respectively. Algorithm
\ref{alg:symmetry-adapted-basis} provides a change of basis matrix $T$
such that every positive semidefinite matrix $Q=(q_{ij})$ that commutes
with $D(\sigma)$ for $\sigma \in S_3$ is of the above form after computing $T^{T} Q T$.
\end{proof}
It has been proved by Scheiderer \cite[Corollary 3.5]{Scheiderer2017sumofsquareslengthofrealforms} that every generic ternary sextic that is SOS admits a representation using four squares; in other words, the corresponding Gram spectrahedron has extreme rays consisting of matrices of rank $4$. Our main theorem in this section establishes four as the minimal rank for generic symmetric ternary sextics that are SOS using the technology of Gr\"obner bases.  
\begin{theorem}  \label{thm:symmetric-ternary-sextic-main}
Let $f \in \R[x_1,x_2,x_3]_6$ be a generic symmetric polynomial. If $f$ is SOS, the symmetry adapted Gram spectrahedron has extreme points consisting
of matrices of rank $4$.
\end{theorem}
\begin{proof}
The polynomial $f$ is parametrized by $7$ coefficients 
which we call $a_1, a_2, \ldots, a_7$. It is also represented
as $f = m(x)^T Q m(x)$ by a $10 \times 10$ symmetric matrix 
$Q = (q_{ij})$.
After equating coefficients and using a symmetry adapted basis we get
a block-diagonal $Q$ where 
$$Q_f^{\ytableausetup{boxsize=.4em}\ydiagram{3}} = \begin{pmatrix}
					a_1+2q_{12} & \sqrt{2}(\frac{a_2}{2}+q_{16}+q_{18}) & \sqrt{3}q_{110} \\
					\sqrt{2}(\frac{a_2}{2}+q_{16}+q_{18}) & \alpha & \sqrt{6}q_{410} \\
					\sqrt{3}q_{110} & \sqrt{6}q_{410} & a_7 - 6q_{49} \\
					\end{pmatrix} $$
					$$Q_f^{\ytableausetup{boxsize=.4em}\ydiagram{2,1}} = \begin{pmatrix}  
						a_1-q_{12} & \frac{\sqrt{2}}{2}(a_2-q_{16}-q_{18}) & \frac{\sqrt{6}}{2}(q_{16}-q_{18})  \\ 
					 \frac{\sqrt{2}}{2}(a_2-q_{16}-q_{18}) & \beta_1 &	\frac{\sqrt{3}}{2}( \frac{a_5}{2} - q_{12}-q_{49})  \\ 
					 \frac{\sqrt{6}}{2}(q_{16}-q_{18}) & \frac{\sqrt{3}}{2}( \frac{a_5}{2} - q_{12}-q_{49}) &  \beta_2  \\
						\end{pmatrix} $$
					$$Q_f^{\ytableausetup{boxsize=.4em}\ydiagram{1,1,1}} = a_3 -\frac{a_4}{2} - \frac{a_5}{2} + a_6 + q_{12} - 2q_{16} - 2q_{18} + q_{110} - q_{49} - 2q_{410}$$
where 
\begin{align*}
\alpha &= a_3 + \frac{a_4}{2} + \frac{a_5}{2} + a_6 - q_{12} - 2q_{16} - 2q_{18} - q_{110} +q_{49}- 2q_{410}\\
\beta_1 &= a_3 + \frac{a_4}{2} - \frac{a_5}{4} - \frac{a_6}{2}+ \frac{q_{12}}{2} - 2q_{16} + q_{18} - q_{110} - \frac{q_{49}}{2}+ q_{410}\\
\beta_2 &= a_3 - \frac{a_4}{2} + \frac{a_5}{4} - \frac{a_6}{2} - \frac{q_{12}}{2} - 2q_{16} + q_{18} + q_{110} + \frac{q_{49}}{2} + q_{410}.
\end{align*}
Now, to get a matrix of rank of $3$, we have four cases:
\begin{enumerate}
    \item[a)] Trivial block has rank $3$ and all other blocks have rank zero.
    \item[b)] Trivial block has rank $2$ and the alternating block has rank $1$.
    \item[c)] Trivial block and standard block have rank $1$ each.
    \item[d)] Standard block and alternating block have rank $1$ each.
\end{enumerate}
In the first case we set all of the linear forms in the standard block and the alternating block to zero and eliminate $q_{ij}$ from the 
ideal generated by these polynomials using a Gr\"{o}bner basis. 
The elimination ideal contains
$$a_5 - 2a_1 - 2a_3 + 2a_2 = 0.$$
This means that a generic symmetric $f$ will not have symmetry adapted
representation of rank $3$ as in the first case. 
The other three cases can be similarly investigated.  
For instance, in the second case we get the following relation on the coefficients:
\begin{multline*}
-10a_1a_2^2 - \frac{5}{4}a_2^3 + 10a_1a_2a_3 - \frac{5}{2}a_2^2a_3 - 4a_1a_3^2 + \frac{5}{2}a_2^2a_4 - 3a_2a_3a_4 + \frac{3}{4}a_2a_4^2 - \frac{1}{2}a_3a_4^2 + 12a_1a_2a_5\\ + \frac{1}{4}a_2^2a_5 - 6a_1a_3a_5 + 3a_2a_3a_5 - 2a_2a_4a_5 + a_3a_4a_5 - \frac{1}{4}a_4^2a_5 - 3a_1a_5^2 + \frac{5}{4}a_2a_5^2 - \frac{1}{2}a_3a_5^2 + \frac{1}{2}a_4a_5^2 \\- \frac{1}{4}a_5^3 - 2a_1a_2a_6 + a_2^2a_6 + 4a_1a_3a_6 + a_2a_4a_6 - a_2a_5a_6 - a_1a_6^2 - 3a_1a_2a_7 - a_2^2a_7 + 2a_1a_3a_7 + a_1a_5a_7=0.
\end{multline*}
The fourth case yields one linear and six cubic relations 
in $a_1, \ldots, a_7$. In the third case, a lengthy computation
in Macaulay 2 \cite{M2} gives a single polynomial of degree $14$ with $6672$ terms. Thus we see that SOS representations with three or fewer squares will only appear in very special cases of symmetric ternary sextics.
\end{proof}
This theorem establishes that we should expect to get a rank four 
SOS representation of symmetric ternary sextics. However, it is important to understand what one would get if an SDP were run on $K_f^{S_3}$. 
This question is related to the geometry of the boundary of $K_f^{S_3}$, and in order to shed some light on this geometry we present some experimental results. 

Figure \ref{fig:ternary-sextic-simulation} is obtained as follows: After generating  $100$ random symmetric ternary sextics, we determined that only 12 of these were SOS according to our numerical SDP returning an optimal solution. For each of these 12 symmetric ternary sextics, we re-ran the SDP for $1000$ distinct, randomly generated linear objective functions. Then we computed the rank of the output matrix by SVD with a cutoff tolerance of $10^{-7}$. Each histogram shows the rank of the optimal matrix.
This and other similar experiments we have conducted show that choosing a random linear functional to minimize resulted most commonly in a solution matrix of rank $6$. However, for some polynomials other ranks were not unusual. For example, for several polynomials, over $100$ of the $1000$ objective functions picked out an optimal solution whose rank was judged to be $4$.
\begin{figure}[H] \label{fig:ternary-sextic-simulation}
\begin{center}
    \includegraphics[width=0.8\textwidth]{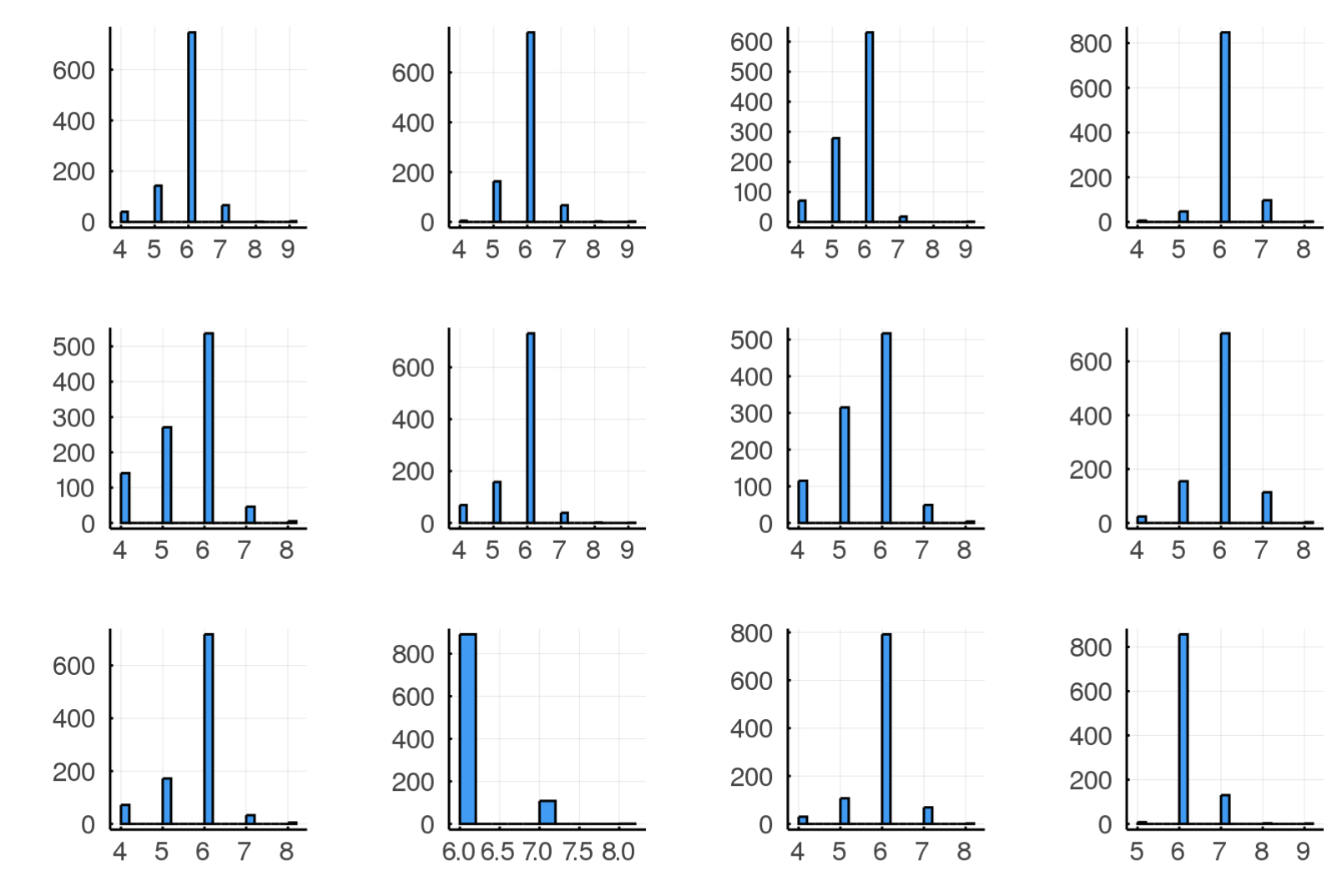}
\end{center}
\caption{Distribution of ranks for SOS decomposition of symmetric ternary
sextics}
\end{figure}

\section{Application to Symmetric Polynomial Inequalities} \label{sec:inequalities}

One application of this machinery is to Muirhead-type inequalities of symmetric polynomials defined on pairs of partitions \cite{CGS2011}.
Let $m_\lambda$, $e_\lambda$, $p_\lambda$, $h_\lambda$, and $s_\lambda$ denote the monomial, elementary, power-sum, homogeneous, and Schur polynomials, respectively, associated to a partition $\lambda$. Given a symmetric polynomial $g(x)$, the \emph{term-normalized symmetric polynomial} is 
$$G(x) := \frac{g(x)}{g(\mathbf{1})}$$
where $g(\mathbf{1})$ is the symmetric polynomial evaluated on the all ones vector.
By $G_\lambda \geq G_\mu$, we mean $G_\lambda(x_1,\ldots,x_n) \geq G_\mu(x_1,\ldots,x_n)$, on the nonnegative orthant.
That is, the inequality holds for any number of variables $n$, but only for $x_i \geq 0$, $i = 1, \ldots, n$.
We denote the term-normalized symmetric polynomials for monomial, elementary, power-sum, homogeneous, and Schur polynomials by $M_\lambda$, $E_\lambda$, $P_\lambda$, $H_\lambda$, and $S_\lambda$, respectively.  

The following theorem is a summary of known results (special cases of which go back to Maclaurin, Muirhead, Newton, and Schur, for example), which are proven in \cite{CGS2011}, \cite{M1902}, and \cite{S2016}.
\begin{theorem} \label{thm:inequalitiesCGS} Let $\lambda$ and $\mu$ be partitions such that $|\lambda| = |\mu|$.  Then
$$\begin{array}{c c l}
M_\lambda \leq M_\mu  &\iff  &\mu \succeq \lambda\\
E_\lambda \leq E_\mu  &\iff  &\lambda \succeq \mu\\
P_\lambda \leq P_\mu  &\iff  &\mu \succeq \lambda\\
S_\lambda \leq S_\mu  &\iff  &\mu \succeq \lambda
\end{array}$$
whereas $\mu \succeq \lambda$ implies that $H_\lambda \leq H_\mu$, i.e.,
$$\begin{array}{c c l}
H_\lambda \leq H_\mu  &\Longleftarrow  &\mu \succeq \lambda\\
\end{array}$$
\end{theorem}

The converse for the homogeneous symmetric functions was conjectured in \cite{CGS2011} in 2011.
In \cite{heatonshankar2019sos}, two authors of the current paper used the  theory of symmetric SOS polynomials to disprove this conjecture by providing a counterexample.
Specifically, $\big(H_{44} - H_{521}\big) (x_1^2,x_2^2,x_3^2)$ is shown to be SOS, thus implying the inequality $H_{44} \geq H_{521}$. This is despite partitions $\mu = (4 \, 4)$ and $ \lambda = (5 \, 2 \, 1)$ being incomparable in the dominance order.

In fact, many counterexamples were found by searching over partitions of 8, 9 and 10.  Below we provide a poset of all differences of term-normalized homogeneous polynomials that are SOS. That is, for each arrow going from $\lambda$ to $\mu$, $\big(H_\mu - H_\lambda\big)(x_1^2,x_2^2,x_3^2)$ is an SOS polynomial.  The black arrows coincide with the dominance order, while the blue arrows are for incomparable pairs of partitions, i.e. each blue arrow is an explicit counterexample to the conjecture.

\begin{figure}[H]
	\centering
		\includegraphics[width=0.8\textwidth]{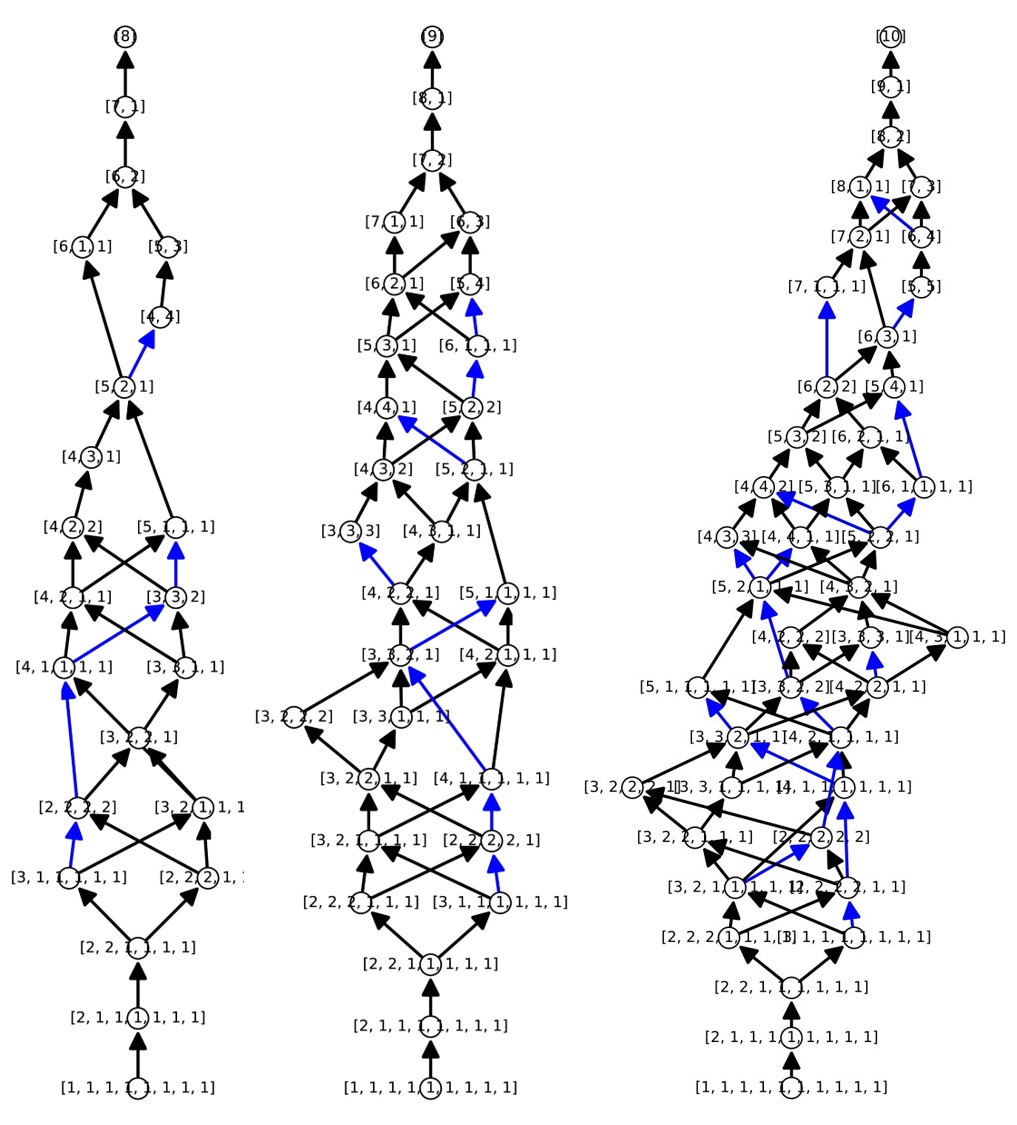}
	\label{fig:HSOS_Poset}
\end{figure}

\bibliographystyle{plain}
\bibliography{references}

\begin{thebibliography}{10}

\bibitem{BGSV2012}
Christine Bachoc, Dion~C. Gijswijt, Alexander Schrijver, and Frank Vallentin.
\newblock Invariant semidefinite programs.
\newblock In {\em Handbook on semidefinite, conic and polynomial optimization},
  volume 166 of {\em Internat. Ser. Oper. Res. Management Sci.}, pages
  219--269. Springer, New York, 2012.

\bibitem{BR2015}
Matthias Beck and Sinai Robins.
\newblock {\em Computing the continuous discretely}.
\newblock Undergraduate Texts in Mathematics. Springer, New York, second
  edition, 2015.
\newblock Integer-point enumeration in polyhedra, With illustrations by David
  Austin.

\bibitem{BGPS-Low-rank-SOS}
Grigoriy Blekherman, Daniel Plaumann, Rainer Sinn, and Cynthia Vinzant.
\newblock Low-rank sum-of-squares representations on varieties of minimal
  degree.
\newblock {\em Int. Math. Res. Not. IMRN}, (1):33--54, 2019.

\bibitem{CasselsEllisonPfister1971}
John W.~S. Cassels, William~J. Ellison, and Albrecht Pfister.
\newblock On sums of squares and on elliptic curves over function fields.
\newblock {\em J. Number Theory}, 3:125--149, 1971.

\bibitem{ChoiLamReznick1995sumsofsquaresofrealpolys}
Man-Duen Choi, Tsit~Yuen Lam, and Bruce Reznick.
\newblock Sums of squares of real polynomials.
\newblock In {\em {$K$}-theory and algebraic geometry: connections with
  quadratic forms and division algebras ({S}anta {B}arbara, {CA}, 1992)},
  volume~58 of {\em Proc. Sympos. Pure Math.}, pages 103--126. Amer. Math.
  Soc., Providence, RI, 1995.

\bibitem{GramSpectrahedra2017chuaplaumannsinnvinzant}
Lynn Chua, Daniel Plaumann, Rainer Sinn, and Cynthia Vinzant.
\newblock Gram spectrahedra.
\newblock In {\em Ordered algebraic structures and related topics}, volume 697
  of {\em Contemp. Math.}, pages 81--105. Amer. Math. Soc., Providence, RI,
  2017.

\bibitem{Coble}
Arthur~B. Coble.
\newblock {\em Algebraic geometry and theta functions}, volume~10 of {\em
  American Mathematical Society Colloquium Publications}.
\newblock American Mathematical Society, Providence, R.I., 1982.
\newblock Reprint of the 1929 edition.

\bibitem{Cotton1990ChemicalApplicationsOfGroupTheory}
F.~Albert Cotton.
\newblock {\em Chemical Applications of Group Theory}.
\newblock John Wiley \& Sons, Inc., 1990.

\bibitem{CGS2011}
Allison Cuttler, Curtis Greene, and Mark Skandera.
\newblock Inequalities for symmetric means.
\newblock {\em European J. Combin.}, 32(6):745--761, 2011.

\bibitem{FS1992}
Albert F\"{a}ssler and Eduard Stiefel.
\newblock {\em Group theoretical methods and their applications}.
\newblock Birkh\"{a}user Boston, Inc., Boston, MA, 1992.
\newblock Translated from the German by Baoswan Dzung Wong.

\bibitem{GP2004}
Karin Gatermann and Pablo~A. Parrilo.
\newblock Symmetry groups, semidefinite programs, and sums of squares.
\newblock {\em J. Pure Appl. Algebra}, 192(1-3):95--128, 2004.

\bibitem{G1998}
Christopher~G. Gibson.
\newblock {\em Elementary geometry of algebraic curves: an undergraduate
  introduction}.
\newblock Cambridge University Press, Cambridge, 1998.

\bibitem{M2}
Daniel~R. Grayson and Michael~E. Stillman.
\newblock Macaulay2, a software system for research in algebraic geometry.
\newblock Available at \url{http://www.math.uiuc.edu/Macaulay2/}.

\bibitem{heatonshankar2019sos}
Alexander {Heaton} and Isabelle {Shankar}.
\newblock {An SOS counterexample to an inequality of symmetric functions},
  2019.
\newblock \href{https://arxiv.org/abs/1909.00081}{arxiv.org/abs/1909.00081}.

\bibitem{Macdonald}
Ian~G. Macdonald.
\newblock {\em Symmetric functions and {H}all polynomials}.
\newblock Oxford Classic Texts in the Physical Sciences. The Clarendon Press,
  Oxford University Press, New York, second edition, 2015.
\newblock With contribution by Andrey V. Zelevinsky and a foreword by Richard
  Stanley, Reprint of the 2008 paperback edition.

\bibitem{2019linearalgebramethod}
Laura Menini, Corrado Possieri, and Antonio Tornamb\`e.
\newblock A linear algebra method to decompose forms whose length is lower than
  the number of variables into weighted sum of squares.
\newblock {\em Internat. J. Control}, 92(11):2647--2666, 2019.

\bibitem{M1902}
Robert~F. Muirhead.
\newblock Some methods applicable to identities and inequalities of symmetric
  algebraic functions of n letters.
\newblock {\em Proceedings of the Edinburgh Mathematical Society},
  21:144–162, 1902.

\bibitem{Parrilo13}
Pablo~A. Parrilo.
\newblock Polynomial optimization, sums of squares, and applications.
\newblock In {\em Semidefinite optimization and convex algebraic geometry},
  volume~13 of {\em MOS-SIAM Ser. Optim.}, pages 47--157. SIAM, Philadelphia,
  PA, 2013.

\bibitem{PSV2011}
Daniel Plaumann, Bernd Sturmfels, and Cynthia Vinzant.
\newblock Quartic curves and their bitangents.
\newblock {\em J. Symbolic Comput.}, 46(6):712--733, 2011.

\bibitem{PRSS2004}
Victoria Powers, Bruce Reznick, Claus Scheiderer, and Frank Sottile.
\newblock A new approach to {H}ilbert's theorem on ternary quartics.
\newblock {\em C. R. Math. Acad. Sci. Paris}, 339(9):617--620, 2004.

\bibitem{GR1995}
Motakuri Ramana and Alan~J. Goldman.
\newblock Some geometric results in semidefinite programming.
\newblock {\em J. Global Optim.}, 7(1):33--50, 1995.

\bibitem{RSST2018}
Annie Raymond, James Saunderson, Mohit Singh, and Rekha~R. Thomas.
\newblock Symmetric sums of squares over {$k$}-subset hypercubes.
\newblock {\em Math. Program.}, 167(2, Ser. A):315--354, 2018.

\bibitem{RST2018}
Annie Raymond, Mohit Singh, and Rekha~R. Thomas.
\newblock Symmetry in {T}ur\'{a}n sums of squares polynomials from flag
  algebras.
\newblock {\em Algebr. Comb.}, 1(2):249--274, 2018.

\bibitem{Robinson1973somepolysnotsumofsquares}
Raphael~M. Robinson.
\newblock Some definite polynomials which are not sums of squares of real
  polynomials.
\newblock In {\em Selected questions of algebra and logic (collection dedicated
  to the memory of {A}. {I}. {M}al'cev) ({R}ussian)}, pages 264--282. 1973.

\bibitem{Sagan}
Bruce~E. Sagan.
\newblock {\em The symmetric group}, volume 203 of {\em Graduate Texts in
  Mathematics}.
\newblock Springer-Verlag, New York, second edition, 2001.
\newblock Representations, combinatorial algorithms, and symmetric functions.

\bibitem{Scheiderer2017sumofsquareslengthofrealforms}
Claus Scheiderer.
\newblock Sum of squares length of real forms.
\newblock {\em Math. Z.}, 286(1-2):559--570, 2017.

\bibitem{Serre1977linearRepresentationsOfFiniteGroups}
Jean-Pierre Serre.
\newblock {\em Linear representations of finite groups}.
\newblock Springer-Verlag, New York-Heidelberg, 1977.
\newblock Translated from the second French edition by Leonard L. Scott,
  Graduate Texts in Mathematics, Vol. 42.

\bibitem{S2016}
Suvrit Sra.
\newblock On inequalities for normalized {S}chur functions.
\newblock {\em European J. Combin.}, 51:492--494, 2016.

\bibitem{Stanley2}
Richard~P. Stanley.
\newblock {\em Enumerative combinatorics. {V}ol. 2}, volume~62 of {\em
  Cambridge Studies in Advanced Mathematics}.
\newblock Cambridge University Press, Cambridge, 1999.
\newblock With a foreword by Gian-Carlo Rota and appendix 1 by Sergey Fomin.

\bibitem{WSV2000}
Henry Wolkowicz, Romesh Saigal, and Lieven Vandenberghe, editors.
\newblock {\em Handbook of semidefinite programming}, volume~27 of {\em
  International Series in Operations Research \& Management Science}.
\newblock Kluwer Academic Publishers, Boston, MA, 2000.
\newblock Theory, algorithms, and applications.

\bibitem{Yiu2001sumofsquareslength}
Paul Yiu.
\newblock The length of {$x^4_1+x^4_2+x^4_3+x^4_4$} as a sum of squares.
\newblock {\em J. Pure Appl. Algebra}, 156(2-3):367--373, 2001.

\end{thebibliography}

\end{document}